\documentclass[11pt]{amsart}

\usepackage{geometry}
\usepackage{colortbl}
\usepackage{amsmath,amssymb}
\usepackage{mathrsfs}
\usepackage{graphicx}
\usepackage{amsmath}
\usepackage{dsfont}
\usepackage{multirow}
\usepackage{graphicx}
\usepackage{epstopdf}
\usepackage{color}
\usepackage{ulem}

\usepackage{enumerate}
\usepackage{caption}
\usepackage{subfig}
\usepackage{bm}
\usepackage{overpic}

\usepackage{tikz}

\usepackage{algorithm}
\usepackage{algorithmic}

\usepackage{varioref}
\usepackage{hyperref}
\usepackage{cleveref}

\makeatletter

\newcommand{\Rmnum}[1]{\expandafter\@slowromancap\romannumeral #1@}
\makeatother

\crefformat{figure}{Fig.~#2#1#3}
\crefmultiformat{figure}{Figures~(#2#1#3)}{ and~(#2#1#3)}{, (#2#1#3)}{ and~(#2#1#3)}
\crefformat{table}{Table~#2#1#3}

\newtheorem{theorem}{Theorem}[section]

\theoremstyle{definition}
\newtheorem{assumption}{Assumption}[section]
\newtheorem{lemma}{Lemma}[section]
\newtheorem{proposition}{Proposition}[section]
\newtheorem{definition}{Ddefinition}[section]

\newtheorem{example}{Example}[section]

\theoremstyle{remark}

\newtheorem{remark}{Remark}

\DeclareMathOperator*{\argmin}{arg\,min}

\numberwithin{equation}{section}
\begin{document}

\title{A quasi-Monte Carlo multiscale method for the wave propagation in random media}

\author{Panchi Li}
\address{Department of Mathematics, The University of Hong Kong, Hong Kong, China}
\email{lipch@hku.hk}

\author{Zhiwen Zhang$^{\ast}$}\thanks{*Corresponding author}
\address{Department of Mathematics, The University of Hong Kong, Hong Kong, P.R. China.  Materials Innovation Institute for Life Sciences and Energy (MILES), HKU-SIRI, Shenzhen, P.R. China. }
\email{zhangzw@hku.hk}

\graphicspath{{figures/}}


\date{\today}

\dedicatory{}
\begin{abstract}
In this paper, we propose and analyze an accurate numerical approach to simulate the Helmholtz problem in a bounded region with a random refractive index, where the random refractive index is denoted using an infinite series parameterized by stochastic variables. To calculate the statistics of the solution numerically, we first truncate the parameterized model and adopt the quasi-Monte Carlo (qMC) method to generate stochastic variables. We develop a boundary-corrected multiscale method to discretize the truncated problem, which allows us to accurately resolve the Robin boundary condition with randomness. The proposed method exhibits superconvergence rates in the physical space (theoretical analysis suggests $\mathcal{O}(H^4)$ for $L^2$-error and $\mathcal{O}(H^2)$ for a defined $V$-error). Owing to the employment of the qMC method, it also exhibits almost the first-order convergence rate in the random space. We provide the wavenumber explicit convergence analysis and conduct numerical experiments to validate key features of the proposed method.
\end{abstract}
\maketitle

\section{Introduction}
\label{sec:introduction}

The Helmholtz equation is a fundamental partial differential equation that describes wave propagation arising in acoustics, electromagnetism, and quantum mechanics~\cite{colton1998inverse}. Efficiently solving the Helmholtz problem with large wavenumber is still among the most challenging tasks in modern scientific computing, because of the highly oscillatory nature and the pollution effect \cite{doi:10.1137/S0036142994269186,Dispersion_pollution1999Deraemaeker,IHLENBURG19959}. Recently, the understanding and quantification of various uncertainties due to the randomness in materials, for wave propagation have drawn many attentions, such as the problems with random refractive index \cite{feng2015efficient,feng2017efficient,ganesh2021quasi,Ma2023CAC}, the random wavenumber \cite{Pulch2024}, as well as the general random model \cite{graham2019helmholtz,graham2025quasi,1144755,MCDANIEL2020132491,8770094}.

In the case of random media, the refractive index can be expressed as random perturbations controlled by a small parameter $\epsilon$ \cite{feng2015efficient,feng2017efficient}. The stochastic solution is then approximated as a series in powers of $\epsilon^j, j = 0, 1, 2, \cdots$. Meanwhile, the coefficients in the series are approximated using the interior-penalty discontinuous Galerkin (IPDG) method owing to its good stability and the capability to reduce the pollution errors by tuning the penalty parameters~\cite{doi:10.1137/080737538,doi:10.1137/21M143827X}. The random refractive index can also be denoted by the infinite series with parameterizations of stochastic variables uniformly distributed in a bounded domain. Numerically, the resulting problem with a finite truncation is further discretized using the classical Galerkin method or the multiscale reduced method \cite{ganesh2021quasi,Ma2023CAC}. To estimate the statistics of the solution, both the traditional Monte-Carlo (MC) method and quasi-Monte Carlo (qMC) methods are widely employed. In contrast to the classical MC method, the qMC method uses carefully chosen points to produce an almost first-order convergence rate in the random space. The high-order convergence rate of qMC means fewer sampling points are required to reach a comparable error. For more information on the qMC method and its applications to diffusion problems, see, for example, \cite{Dick_Kuo_Sloan_2013,graham2015quasi,kuo2016application,doi:10.1137/110845537}.

The standard Galerkin method suffers from the severe dependence of the mesh size on the wavenumber, because of the well-known instability issue, the highly oscillatory of the solution, and the pollution effect. To obtain the accurate solution of the Helmholtz equation with constant or low-contrast coefficients, a number of approaches have been designed and analyzed, such as $hp$-finite elements~\cite{85557aa6-cec7-38fa-8b66-039de787274f,doi:10.1137/090776202}, discontinuous Galerkin methods~\cite{doi:10.1137/080737538,5fdb831a-08b0-3ff0-a756-5c4f22c7272d,doi:10.1137/090761057,https://doi.org/10.1002/nme.1575} and multiscale methods. In particular, the localized orthogonal decomposition (LOD) method has been proven to be effective in eliminating the pollution effect with a loose dependence between the wavenumber and mesh size~\cite{doi:10.1137/21M1414607,fda66b15-cbd8-3cbb-a160-13398a1c4325}. Thereafter, a novel localization strategy, named the super-LOD, is proposed to reduce the localization error \cite{doi:10.1137/21M1465950,hauck2023super}. Besides, there are also other multiscale methods such as the generalized multiscale finite element method \cite{doi:10.1137/22M1507802,fu2024edgemultiscalespacebased,doi:10.1137/19M1267180,doi:10.1137/21M1466748}, the heterogeneous multiscale method \cite{doi:10.1137/16M1108820}, and the multiscale hybrid-mixed method \cite{doi:10.1137/19M1255616}.  Multiscale methods have excellent capabilities to eliminate the pollution effect and deal with the discontinuous coefficients in a low-dimensional coarse space. Nevertheless, there is rarely related work to achieve an optimal approximation for the random Helmholtz problem by combining multiscale and qMC methods.

The main objectives of this paper are (i) to provide an accurate multiscale method to simulate the Helmholtz problem in random media; (ii) to carry out a rigorous error analysis for the proposed method. More precisely, we approximate the random refractive index using the infinite series parameterized by the stochastic variables. Numerically, the infinite-dimensional parameterized model is truncated and the qMC method is employed to generate stochastic variables. Since the Robin boundary condition also depends on the stochastic variables in a nonlinear form, we develop a multiscale method with a boundary corrector for the resulting problem for every realization of the random refractive index. We remark that the referred multiscale method is a LOD method, in which we construct the basis by solving a set of optimal problems.
When the Dirichlet boundary and Neumann boundary conditions are considered, the LOD method with boundary correctors has been proposed in \cite{doi:10.1137/130933198}. And a merged method of LOD and operator-adapted wavelets (gamblets) has been proposed \cite{doi:10.1137/21M1414607} to effectively solve the Helmholtz problem with various boundary conditions.
For the Helmholtz problem with the nontrivial Robin boundary condition, the solution shall belong to a complex-valued Sobolev space. Hence, similar to the LOD and super-LOD methods \cite{doi:10.1137/21M1465950,fda66b15-cbd8-3cbb-a160-13398a1c4325}, here the complex-valued multiscale basis is also constructed.

Therefore, the main contributions of this paper are an accurate numerical approach that combines the multiscale method and qMC methods, and a rigorous error analysis, especially for the wavenumber explicit estimate. In the framework of the LOD method, the pollution effect is naturally eliminated. Error estimates further show that the $L^2$-error can reach a fourth-order convergence rate, while a defined $V$-error (equality to $H^1$-norm but depends on the wavenumber) can reach a second-order convergence rate. Furthermore, owing to the qMC samples being used, it exhibits an almost first-order convergence rate in the random space. Some numerical experiments are carried out to validate the convergence. Especially for the 1D problem, the optimal $V$-error can reach a fourth-order convergence rate as well. Furthermore, we conduct numerical tests on heterogeneous media, demonstrating the superiority of boundary correction. The techniques developed in this paper can be used to construct accurate numerical schemes in other complex-valued problems.


The rest of this paper is organized as follows. In \Cref{sec:model-wavenumber-explicit-estimate}, we introduce the Helmholtz problem and derive a wavenumber explicit estimate. In \Cref{sec:msfem-phsical-space}, we propose the multiscale method with a boundary corrector for the parametric Helmholtz problem. In \Cref{sec:regularity-wrt-variables}, we estimate the regularity of the solution with respect to the stochastic variable. In \Cref{sec:dimension-truncation}, we present the truncation error and the qMC error. The overall numerical algorithm and the total error approximation are given in \Cref{sec:overall-method-analysis}. Numerical results are depicted in \Cref{sec:numerics}. And conclusions are drawn in \Cref{sec:conclusions}.

\section{The problem and wavenumber explicit estimate}\label{sec:model-wavenumber-explicit-estimate}
\subsection{Notation and preliminaries}
Define the inner products
\begin{equation*}
  (u, v) = \int_D u \bar{v} \mathrm{d}\mathbf{x}, \quad \langle u, v\rangle = \int_{\partial D} u \bar{v} \mathrm{d}s,
\end{equation*}
with $\bar{v}$ the conjugate of the complex number $v$. In the sequel, we use $L^2(D)$ and $L^2(\partial D)$ abbreviate the canonical $L^2(D; \mathds{C})$ and $L^2(\partial D; \mathds{C})$ spaces, respectively.
The induced $L^2$-norms are $\|u\|^2_{L^2(D)} = (u, u)$ and $\|u\|^2_{L^2(\partial D)} = \langle u, u\rangle$. Meanwhile, we define the $\kappa$-weighted norm
\begin{equation*}
  \|u\|^2_{V} = \|\nabla u\|^2_{L^2(D)} + \kappa^2 \|u\|^2_{L^2(D)}.
\end{equation*}
Throughout this paper, the notation $a \lesssim b$ denotes $a \leq Cb$ with the constant $C$ independent of the wavenumber $\kappa$, and the mesh sizes $H$ and $h$.

\subsection{Model problem}
In a bounded domain with a random refractive index and spatial heterogeneity, the wave propagation is modeled by the Helmholtz problem:
\begin{align}
  &-\Delta u - \kappa^2n(\mathbf{x}, \boldsymbol{\omega})u = f(\mathbf{x}), \quad &\mathbf{x} \in D, \; \boldsymbol{\omega} \in \Omega,
  \label{equ:Helmholtz-equation} \\
  &\nabla u \cdot \nu - i\kappa\sqrt{n(\mathbf{x}, \boldsymbol{\omega})}u = 0, \quad & \mathbf{x} \in \partial D, \; \boldsymbol{\omega} \in \Omega,
  \label{equ:robin-boundary}
\end{align}
where $D \subset \mathds{R}^d (d = 1,2,3)$ is a convex bounded polygonal domain with the boundary $\partial D$, $\nu$ is the unit outward normal vector on $\partial D$, $\kappa$ is the wavenumber and $n(\mathbf{x}, \boldsymbol{\omega})$ is the random refractive index.
For each $\mathbf{x} \in D$, the refractive index $n(\mathbf{x}, \cdot)$ defines a positive-valued random variable over $\Omega$.

Consider the random refractive index parameterized by an infinite-dimensional vector $\boldsymbol{\omega} = (\omega_1, \omega_2, \cdots)$ with the form
\begin{equation}
  n(\mathbf{x}, \boldsymbol{\omega}) = n_0(\mathbf{x}) + \sum_{j=1}^{\infty} \omega_j\psi_j(\mathbf{x}),
  \label{equ:refractive-index-infinite-series}
\end{equation}
in which we assume that the parameter $\boldsymbol{\omega}$ is uniformly distributed on
$$\Omega := \left[-\frac 12, \frac12\right]^{\mathds{N}}$$
with $\mathds{N}$ being a set of positive integers. The $\psi_j(\mathbf{x})$ are given functions and can belong to the
Karhunen-Lo\'eve (KL) eigensystem of a covariance operator in $L^2(D, \mathds{R})$.

The variational formulation for \eqref{equ:Helmholtz-equation}-\eqref{equ:robin-boundary} reads: for each $\boldsymbol{\omega} \in \Omega$ find $u(\mathbf{x}, \boldsymbol{\omega}) \in L^2(\Omega, H^1(D))$ such that
\begin{equation}
  a(\boldsymbol{\omega}; u, v) = F(v), \quad \forall v \in H^1(D),
  \label{equ:definition-weak-form}
\end{equation}
where the sesquilinear form $a(\boldsymbol{\omega}; \cdot, \cdot): H^1(D) \times H^1(D) \rightarrow \mathds{C}$ and conjugate form $F: H^1(D) \rightarrow \mathds{C}$ are defined as
\begin{gather}
  a(\boldsymbol{\omega}; u, v) = \int_D \nabla u \cdot \nabla \bar{v} \mathrm{d}\mathbf{x} - \kappa^2 \int_D n(\mathbf{x}, \boldsymbol{\omega}) u \bar{v} \mathrm{d}\mathbf{x} - i\kappa \int_{\partial D} \sqrt{n(\mathbf{x}, \boldsymbol{\omega})} u \bar{v} \mathrm{d}s,\label{equ:weak-form-a}\\
  F(v) = \int_D f \bar{v} \mathrm{d}\mathbf{x}, \quad \forall  u, v \in H^1(D).
\end{gather}

\subsection{Wavenumber explicit estimates}
The wavenumber explicit stability estimates usually play a pivotal role in the development of numerical methods. Before the formal estimates are given, we introduce some definitions and assumptions on the domain and the refractive index.

\begin{definition}
  We introduce two definitions for the problem.
  \begin{itemize}
    \item Star-shaped. The set $D \subset \mathds{R}^d$ is star-shaped with respect to the point $\mathbf{x}_D \in D$ if for all $\mathbf{x} \in D$ the line segment $[\mathbf{x}_D, \mathbf{x}] \subset D$.
    \item Denote $\alpha = \mathbf{x} - \mathbf{x}_D$. The refractive index satisfies $n(\mathbf{x}, \boldsymbol{\omega}) + \nabla n(\mathbf{x}, \boldsymbol{\omega}) \cdot \alpha \geq \mu > 0$ for all $\boldsymbol{\omega} \in \Omega$.
  \end{itemize}
\end{definition}
Assume the domain $D$ is star-shaped, i.e., there exists a point $\mathbf{x}_D \in D$ such that
\begin{equation}
  \beta_1 \leq \alpha^T \cdot \nu \leq \beta_2, \quad \forall\; \mathbf{x} \in D, \label{equ:condition-star-shape}
\end{equation}
where $\beta_1$ and $\beta_2$ are two finite positive constants. Meanwhile, for every stochastic variable $\boldsymbol{\omega} \in \Omega$, we assume
\begin{equation}
  0 < n_1 \leq n(\mathbf{x}, \boldsymbol{\omega}) \leq n_2, \quad \forall \mathbf{x} \in D.
  \label{equ:bounds-for-n}
\end{equation}
Besides, two equalities that will be repeatedly used in the following proof~\cite{cummings2006sharp,cummings2001analysisPhDThesis}:
\begin{gather*}
  2\Re(v\alpha\cdot\nabla \bar{v}) = \alpha \cdot \nabla |v|^2, \\
  2\Re(\nabla v \cdot \nabla (\alpha \cdot \nabla \bar{v})) = 2|\nabla v|^2 + \alpha \cdot \nabla |\nabla v|^2.
\end{gather*}
We now derive the stability for the solution of problem \eqref{equ:Helmholtz-equation}-\eqref{equ:robin-boundary}. The explicit dependence of the stability constants on the wavenumber $\kappa$ will be given.
\begin{proposition}
  Let $D$ be a star-shaped domain. For all $\boldsymbol{\omega} \in \Omega$, assume the refractive index admits the assumption \eqref{equ:bounds-for-n} and $\kappa > 0$. We have
  \begin{equation}
    \kappa \| u \|_{L^2(D)} \leq C (1 + \kappa^{-1} )\|f\|_{L^2(D)},
  \end{equation}
  and
  \begin{equation}
    \| u \|_V \leq C (1 + \kappa^{-1}) \|f\|_{L^2(D)},
  \end{equation}
  where $C$ is independent of $\kappa$, but depends on the domain, $n_j$, $\beta_j$ ($j=1,2$) and $\mu$.
  \label{prop:stability-estimate}
\end{proposition}
\begin{proof}
  Let $v = u$ in the weak form \eqref{equ:definition-weak-form}, and we obtain
  \begin{gather}
    \kappa \langle \sqrt{n} u, u \rangle \leq \left| F(u) \right|, \\
    (\nabla u, \nabla u) - \kappa^2(n u, u) \leq \left| F(u) \right|.
  \end{gather}
  We temporarily use $n$ denote $n(\mathbf{x}, \boldsymbol{\omega})$. Employing Cauchy's inequality yields
  \begin{equation*}
    \left| F(u) \right| \leq \|f\|_{L^2(D)} \|u\|_{L^2(D)} \leq \frac{\epsilon_0}{2}\|f\|^2_{L^2(D)} + \frac 1{2\epsilon_0}\|u\|^2_{L^2(D)}.
  \end{equation*}
  Owing to $0 < n_1 \leq n \leq n_2$, the boundary estimate satisfies
  \begin{equation*}
    \| u \|^2_{L^2(\partial D)} \leq \frac{1}{\kappa\sqrt{n_1}} \|f\|_{L^2(D)} \|u\|_{L^2(D)}.
  \end{equation*}
  Let $v = \alpha \cdot \nabla u$ in \eqref{equ:definition-weak-form}. Due to
  \begin{align*}
    \Re \int_D \nabla u \cdot \nabla(\alpha \cdot \nabla \bar{u}) \mathrm{d}\mathbf{x} &= \int_D |\nabla u|^2 \mathrm{d}\mathbf{x} + \frac 12\int_D \alpha\cdot\nabla |\nabla u|^2 \mathrm{d}\mathbf{x} \\
    &= \left( 1 - \frac{d}{2} \right)\int_D |\nabla u|^2 \mathrm{d}\mathbf{x} + \frac 12 \int_{\partial D} (\alpha \cdot \nu)|\nabla u|^2 \mathrm{d}s,
  \end{align*}
  and
  \begin{align*}
    & \Re \int_{D} n u (\alpha \cdot \nabla \bar{u}) \mathrm{d}\mathbf{x}
    = \frac 12\int_{\partial D} n (\alpha \cdot \nu) |u|^2 \mathrm{d}s - \frac 12\int_D \nabla n \cdot \alpha |u|^2 \mathrm{d}\mathbf{x} - \frac {d}{2}\int_D n |u|^2 \mathrm{d}\mathbf{x},
  \end{align*}
  we obtain
  \begin{align}
    \Re\int_D f(\alpha\cdot \nabla\bar{u}) \mathrm{d}\mathbf{x} = \left(1-\frac{d}{2}\right)\|\nabla u\|_{L^2(D)}^2 + \frac {d\kappa^2}{2}\int_D n |u|^2 \mathrm{d}\mathbf{x} + \frac {\kappa^2}2\int_D \nabla n \cdot \alpha |u|^2 \mathrm{d}\mathbf{x} \nonumber\\
    -\Re\int_{\partial D} i\kappa \sqrt{n} u \alpha\cdot\nabla \bar{u}\mathrm{d}s + \frac 12\int_{\partial D} (\alpha \cdot \nu) |\nabla u|^2 \mathrm{d}s
    - \frac {\kappa^2}2\int_{\partial D} n (\alpha \cdot \nu) |u|^2 \mathrm{d}s .
    \label{equ:integral-for-L2-H1}
  \end{align}
  And then we arrive at
  \begin{align*}
    &\frac{\mu\kappa^2}{2}\| u \|_{L^2(D)} + \frac{\beta_1}{2} \|\nabla u\|_{L^2(\partial D)} \\
    \leq &\frac {\kappa^2}2\int_D \nabla n(\mathbf{x}, \boldsymbol{\omega}) \cdot \alpha |u|^2 \mathrm{d}\mathbf{x} + \frac {\kappa^2}{2}\int_D n(\mathbf{x}, \boldsymbol{\omega}) |u|^2 \mathrm{d}\mathbf{x} + \frac 12\int_{\partial D} (\alpha \cdot \nu) |\nabla u|^2 \mathrm{d}s \\
    = &\Re\int_D f(\alpha\cdot \nabla\bar{u}) \mathrm{d}\mathbf{x} + \Re\int_{\partial D} i\kappa \sqrt{n(\mathbf{x}, \boldsymbol{\omega})} u \alpha\cdot\nabla \bar{u}\mathrm{d}s + \frac {\kappa^2}2\int_{\partial D} n(\mathbf{x}, \boldsymbol{\omega}) (\alpha \cdot \nu) |u|^2 \mathrm{d}s \\
    & + \left(\frac{d}{2}-1\right) \left( \|\nabla u\|_{L^2(D)}^2 - \kappa^2(n(\mathbf{x}, \boldsymbol{\omega})u, u)_{L^2(D)} \right) \\
     \leq & \frac{M}{2\epsilon_1}\|\nabla u\|^2_{L^2(D)} + \frac{M\kappa\sqrt{n_2}}{2\epsilon_2}\|\nabla u\|^2_{L^2(\partial D)} + \frac{d-2}{4\epsilon_0}\|u\|^2_{L^2(D)} \\
     &+ \left( M\kappa\sqrt{n_2} \frac{\epsilon_2}{2} + \frac{\kappa^2\beta_2n_2}{2} \right)\|u\|^2_{L^2(\partial D)} + \left( \frac{M\epsilon_1}{2} + \frac{(d-2)\epsilon_0}{4} \right) \|f\|^2_{L^2(D)},
  \end{align*}
  in which we use the Cauchy inequalities with the parameters $\epsilon_1$ and $\epsilon_2$.
  Since
  \begin{equation*}
    \|\nabla u\|_{L^2(D)}^2 \leq \left(\kappa^2n_2 + \frac 1{2\epsilon_0}\right)\|u\|_{L^2(D)} + \frac{\epsilon_0}{2}\|f\|^2_{L^2(D)},
  \end{equation*}
  we choose $\epsilon_2 = \frac{2M\sqrt{n_2}\kappa}{\beta_1}$ and get
  \begin{align*}
    \frac{\mu\kappa^2}{2}\| u \|_{L^2(D)} \leq & \left\{ \frac{Mn_2\kappa^2}{\epsilon_1} + \frac{d-2}{2\epsilon_0} + \frac{n_2\kappa}{\sqrt{n_1}}\left( \frac{2M^2}{\beta_1} + \frac{\beta_2}{2} \right) \frac 1{\epsilon_3} \right\} \|u\|^2_{L^2(D)}\\
    &+ \left\{\frac{M\epsilon_0}{\epsilon_1} + \frac{n_2\kappa\epsilon_3}{\sqrt{n_1}}\left( \frac{2M^2}{\beta_1} + \frac{\beta_2}{2} \right) + M\epsilon_1 + \frac{d-2}{2}\epsilon_0 \right\} \|f\|^2_{L^2(D)}.
  \end{align*}
  Choose $\epsilon_0 = \frac{d-2}{2C_0\mu\kappa^2}$, $\epsilon_1 = \frac{Mn_2}{C_1\mu}$, $\epsilon_3 = \frac{n_2}{C_3\sqrt{n_1}\mu\kappa}\left( \frac{2M^2}{\beta_1} + \frac{\beta_2}{2} \right)$, and we arrive at
  \begin{multline*}
    \frac{\mu\kappa^2}{2}\| u \|^2_{L^2(D)} \leq (C_0 + C_1 + C_3)\mu\kappa^2 \|u\|^2_{L^2(D)} \\
    + \left( \frac{(d-2)C_1}{2n_2C_0\kappa^2} + \frac{n_2^2}{C_3n_1\mu}\left( \frac{2M^2}{\beta_1} + \frac{\beta_2}{2} \right)^2 + \frac{M^2n_2}{C_1\mu} + \frac{(d-2)^2}{4C_0\mu\kappa^2} \right) \|f\|^2_{L^2(D)}.
  \end{multline*}
  Let $C_0 + C_1 + C_3 < \frac 12$, and then there exists a constant $C$ independent of $\kappa$ such that
  \begin{equation*}
    \kappa^2 \| u \|^2_{L^2(D)} \leq C (1 + \kappa^{-2} )\|f\|^2_{L^2(D)},
  \end{equation*}
  or equivalently, $\kappa \| u \|_{L^2(D)} \leq C (1 + \kappa^{-1} )\|f\|_{L^2(D)}$.
  Then, we obtain
  \begin{align*}
    \|\nabla u\|^2_{L^2(D)} + \kappa^2\| u \|^2_{L^2(D)} \leq &\kappa^2 (n_2 + 1)\| u \|^2_{L^2(D)} + \frac 1{\epsilon_0}\| u \|^2_{L^2(D)} + \epsilon_0 \|f\|^2_{L^2(D)} \\
    \leq &C(n_2 + 2)(1 + \kappa^{-2} )\|f\|^2_{L^2(D)} + \kappa^{-2}\|f\|^2_{L^2(D)},
  \end{align*}
  which infers
  \begin{equation*}
    \|u\|_V = \left(\|\nabla u\|^2_{L^2(D)} + \kappa^2\| u \|^2_{L^2(D)}\right)^{\frac 12} \leq C(1 + \kappa^{-1})\|f\|_{L^2(D)}.
  \end{equation*}
  This completes the proof.
\end{proof}

\begin{remark}
    Once the estimate $\|u\|_{L^2(D)}$ is given, we also have
    \begin{equation}
        \|u\|^2_{L^2(\partial D)} \leq C\kappa^{-1}(1 + \kappa^{-2})\|f\|^2_{L^2(D)}.
        \label{equ:bound-u-L2-boundary}
    \end{equation}
    The estimate $\|u\|_{L^2(D)}$ depends on $\|u\|_{L^2(\partial D)}$ although it is not required explicitly. Due to the $\kappa$ dependent coefficient always present, this term also needs to be treated carefully.
\end{remark}

We next declare the well-posedness of the problem.
\begin{lemma}
  Let $\boldsymbol{\omega} \in \Omega$, and the sesquilinear of \eqref{equ:weak-form-a} satisfies
  \begin{equation}
    \inf_{0 \neq u \in H^1(D)} \sup_{0 \neq w \in H^1(D)} \frac{\Re a(u, w)}{\|u\|_V\|w\|_V} \geq \frac{1}{C_{s}(1 + \kappa)},
    \label{equ:coercive-Helmholtz}
  \end{equation}
  where $C_s$ is independent of $\kappa$.
\end{lemma}
\begin{proof}
  Given $u \in H^1(D)$, define $z \in H^1(D)$ as the solution of
  \begin{equation*}
    a(v, z) = 2\kappa^2(n(\mathbf{x}, \boldsymbol{\omega})v, u), \quad \text{for all } v \in H^1(D).
  \end{equation*}
  \cref{prop:stability-estimate} implies $\|z\|^2_V \leq C\kappa^4 \|u\|^2_{L^2(D)}$ (replace $f$ by $2\kappa^2n(\boldsymbol{\omega}, \mathbf{x})u$ in \eqref{equ:definition-weak-form}) and $w = u + z$ satisfies
  \begin{equation*}
    \Re a(u, w) = \Re a(u, u) + \Re a(u, z) = \|\nabla u\|^2_{L^2(D)} + \kappa^2(n u, u).
  \end{equation*}
  Meanwhile, since
  \begin{equation*}
    \|w\|^2_V \leq \|u\|^2_V + \|z\|^2_V \leq \|u\|^2_V + C\kappa^4 \|u\|^2_{L^2(D)} \leq C(1 + \kappa^2) \|u\|^2_V,
  \end{equation*}
  or equivalently, $\|w\|_V \leq C(1 + \kappa) \|u\|_V$. We therefore have
  \begin{equation*}
    \Re a(u, w) = \|\nabla u\|^2_{L^2(D)} + \kappa^2(n u, u) \geq C \|u\|^2_V \geq C_s (1 + \kappa)^{-1}\|u\|_V \|w\|_V,
  \end{equation*}
  which concludes the proof of \eqref{equ:coercive-Helmholtz}.
\end{proof}

\section{Construction of the multiscale basis}
\label{sec:msfem-phsical-space}
\subsection{Two level finite element meshes and quasi-interpolation operator}
This part recalls briefly the notions of finite element meshes and patches, standard finite element spaces and the corresponding quasi-interpolation operator for the discretization.

Let $\mathcal{T}_H$ denote a regular finite element mesh of $D$ with the mesh size $H$. The first-order conforming finite element space on mesh $\mathcal{T}_H$ is given by
\begin{equation*}
  V_H = \{v \in H^1(D)\; |\; \forall T \in \mathcal{T}_H, v|_T \text{ is a polynomial of total degree } \leq 1\}.
\end{equation*}
Denote $\mathcal{N}_H$ the set of all vertices of $\mathcal{T}_H$, and $|\mathcal{N}_H| = N_H$. At the node $\mathbf{x}_j \in \mathcal{N}_H$, the nodal basis function $\phi_j^H(\mathbf{x}) \in V_H$ is determined by $\phi^H_j(\mathbf{x}_k) = \delta_{jk}$. In many practices, the multiscale basis cannot be expressed explicitly, and hence it is approximated using the finite element basis on a refined mesh $\mathcal{T}_h$ with the mesh size $h (< H)$.


For numerical analysis, we introduce a weighted Cl\'ement-type quasi-interpolation operator $\mathcal{I}_H$: $H^1(D) \rightarrow V_H$  \cite{doi:10.1137/S003614299732334X}. Given $v \in H^1(D)$, its definition reads as
\begin{equation}
  \mathcal{I}_H(v) = \sum_{j=1}^{N_H} \frac{(v, \phi_j^H)}{(1, \phi_j^H)}\phi_j^H(\mathbf{x}).
\end{equation}
The interpolation operator $\mathcal{I}_H$ is local stable and the local approximation satisfies for all $v \in H^1(D)$ and $T \in \mathcal{T}_H$
\begin{equation*}
  H_T^{-1}\|v - \mathcal{I}_H(v)\|_{L^2(T)} + \|\nabla (v - \mathcal{I}_H(v))\|_{L^2(T)} \leq C_{\mathcal{I}_H} \|\nabla v\|_{L^2(D_1)}.
\end{equation*}
In this work, the approximation error around boundaries is sufficiently considered. Hence, we refer to the global approximation of the interpolation operator $\mathcal{I}_H$~\cite{doi:10.1137/21M1464154,ern:hal-03226049}: for all $v \in H^{k+1}(D)$
\begin{equation}
  \|v - \mathcal{I}_H(v)\|_{L^2(D)} + H\|\nabla(v - \mathcal{I}_H(v))\|_{L^2(D)} + H^{\frac 12}\|v - \mathcal{I}_H(v)\|_{L^2(\partial D)} \leq C H^{k+1}|v|_{H^{k+1}(D)},
  \label{equ:approximation-interpolation-operator}
\end{equation}
where $|v|_{H^{k+1}(D)} = \|\nabla^{k+1} v\|_{L^2(D)}$.

\subsection{Localized multiscale basis functions}
To resolve the highly oscillatory waves, there is a necessary condition of the dependence between the mesh width $H$ and the wavenumber $\kappa$.
\begin{assumption}
  Given the wave number $\kappa$, we assume that the mesh size $H$ satisfies $H \kappa \lesssim 1$.
  \label{assump:mesh-size-wavenumber}
\end{assumption}

The solution of the Helmholtz problem \eqref{equ:Helmholtz-equation}-\eqref{equ:robin-boundary} is complex-valued because of the nontrivial Robin boundary condition. This implies that the operator-adaptive basis functions are complex-valued as well, ensuring that the multiscale method achieves optimal approximation, similar to \cite{doi:10.1137/21M1465950,doi:10.1137/21M1414607}. Here we present an approach to construct the complex-valued basis using the optimal approach developed in \cite{Hou2017}.

\subsubsection{Global multiscale basis}
For any given $\boldsymbol{\omega}$, denote
$$a_0(\boldsymbol{\omega}; u, v) = \int_D\nabla u \cdot \nabla \bar{v} \mathrm{d}\mathbf{x} - \kappa^2 \int_D n(\mathbf{x}, \boldsymbol{\omega})u \bar{v} \mathrm{d}\mathbf{x},$$
and $a(\boldsymbol{\omega}; u, v) = a_0(\boldsymbol{\omega}; u, v) - i\kappa \langle \sqrt{n(\mathbf{x}, \boldsymbol{\omega})} u, v \rangle$.
Let $u = u_R + iu_I$, and it holds
\begin{equation*}
  \nabla u_R \cdot \nu = -\kappa\sqrt{n(\mathbf{x}, \boldsymbol{\omega})}u_I, \quad \nabla u_I \cdot \nu = \kappa\sqrt{n(\mathbf{x}, \boldsymbol{\omega})}u_R, \quad \forall \mathbf{x} \in \partial D,
\end{equation*}
which indicates that
\begin{equation*}
  \Re (\nabla u \cdot \nu) {u} = -2\kappa\sqrt{n(\mathbf{x}, \boldsymbol{\omega})}u_Ru_I, \quad  \Im (\nabla u \cdot \nu) {u} = \kappa\sqrt{n(\mathbf{x}, \boldsymbol{\omega})} (u_R^2 - u_I^2).
\end{equation*}
Numerical tests demonstrate that better results can be provided when the second term is considered in the construction of the multiscale basis. More precisely, we define a new functional
\begin{equation*}
  \hat{a}(\boldsymbol{\omega}; u, v) = a_0(\boldsymbol{\omega}; u, v) + \kappa \int_{\partial D}\sqrt{n(\mathbf{x}, \boldsymbol{\omega})} (\Re{u}\Re{v} - \Im{u}\Im{v}) \mathrm{d}s.
\end{equation*}
Denote the operator-adaptive basis function $\phi = \phi_R + i\phi_I$. We rewrite it into $\Phi = (\phi_R, \phi_I)^T$ and then the multiscale basis function at $\mathbf{x}_j$ solves the optimal problem:
\begin{equation}
  \begin{aligned}
    &\argmin_{\Phi\in (H^1(D; \mathds{R}))^2} \hat{a}(\Phi, \Phi), \\
    &\mathrm{s.t.} \quad (\Phi, \phi_k^H) = \alpha_j\delta_{jk}(1, 1)^T, \quad k = 1, \cdots, N_H,
  \end{aligned}
  \label{equ:optimal-problems-global}
\end{equation}
where $\alpha_j$ are determined by the weighted Cl\'ement-type quasi-interpolation operator as in \cite{li2025efficientfiniteelementmethods,lizhang2025model}. Note that here
\begin{equation*}
  \hat{a}(\boldsymbol{\omega}; \Phi, \Phi) = \kappa\int_{\partial D} \sqrt{n(\mathbf{x}, \boldsymbol{\omega})}(\phi_R^2 - \phi_I^2) \mathrm{d}s + \int_D |\nabla \Phi|^2 \mathrm{d}\mathbf{x} - \kappa^2\int_D n(\mathbf{x}, \boldsymbol{\omega}) |\Phi|^2 \mathrm{d}\mathbf{x}.
\end{equation*}
Solving the optimal problems \eqref{equ:optimal-problems-global} for $j = 1, \cdots, N_H$ yields the global multiscale basis $\Psi_H = \mathrm{span} \{\Phi_1, \cdots, \Phi_{N_H}\}$. Hence, $\hat{a}(\cdot, \cdot)$ defines an orthogonal decomposition of $H^1(D) = \Psi_H \oplus W$, i.e.,
\begin{equation*}
  \hat{a}(\boldsymbol{\omega}; v_H, w) = 0, \quad \forall \boldsymbol{\omega} \in \Omega, \; v_H \in \Psi_H, \; w \in W.
\end{equation*}
Now we can solve the numerical solution $u_H$ that satisfies
\begin{equation}
  a(\boldsymbol{\omega}; u_H, v_H) = f(v_H), \quad \forall v_H \in \Psi_H.
  \label{equ:discrete-weak-from-msfem}
\end{equation}

\begin{lemma}
  Let \cref{assump:mesh-size-wavenumber} hold. Then we have the inf-sup condition
  \begin{equation*}
    \inf_{0 \neq u_H \in \Psi_H} \sup_{0 \neq w_H \in \Psi_H} \frac{\Re a(u_H, w_H)}{\|u_H\|_V\|w_H\|_V} \geq \frac{1}{C_{s}(1 + \kappa)}.
  \end{equation*}
\end{lemma}
\begin{proof}
  For all $v \in H^1(D)$ and $v_H \in \Psi_H$,
  \begin{equation*}
    a(\boldsymbol{\omega}; u_H, v_H) = \hat{a}(\boldsymbol{\omega}; u_H, v_H) - \langle i\kappa\sqrt{n(\mathbf{x}, \boldsymbol{\omega})}u_H, v_H \rangle - \Im \langle i\kappa\sqrt{n(\mathbf{x}, \boldsymbol{\omega})}\bar{u}_H, v_H \rangle.
  \end{equation*}
  Owing to $\hat{a}(\boldsymbol{\omega}; u_H, v_H) = \hat{a}(\boldsymbol{\omega}; u_H, v)$ and
  \begin{equation*}
    a(\boldsymbol{\omega}; u_H, v_H) = a(\boldsymbol{\omega}; u_H, v) + i\kappa \langle \sqrt{n}u_H, v-v_H \rangle + \Im \langle i\kappa\sqrt{n}\bar{u}_H, v-v_H \rangle
  \end{equation*}
  for all $v - v_H \in W$, and both $u_H, \bar{u}_H $ belong to $\Psi_H$. Then for all $v \in H^1(D)$
  \begin{equation*}
    \Re a(\boldsymbol{\omega}; u_H, v_H) = \Re a(\boldsymbol{\omega}; u_H, v) - \Im \kappa \langle \sqrt{n}u_H, v-v_H \rangle + \Re \kappa \langle \sqrt{n}\bar{u}_H, v-v_H \rangle,
  \end{equation*}
  which infers that
  \begin{equation*}
    \sup_{0 \neq v_H \in \Psi_H} \frac{\Re a(u_H, v_H)}{\|u_H\|_V\|v_H\|_V} \geq \frac{\Re a(u_H, v)}{\|u_H\|_V\|v\|_V}.
  \end{equation*}
  We therefore obtain
  \begin{align*}
    \inf_{0 \neq u_H \in \Psi_H} \sup_{0 \neq v_H \in \Psi_H} \frac{\Re a(u_H, v_H)}{\|u_H\|_V\|v_H\|_V} & \geq \inf_{0 \neq u_H \in \Psi_H}\frac{\Re a(u_H, v)}{\|u_H\|_V\|v\|_V} \\
    & \geq \inf_{0 \neq u \in H^1(D)} \sup_{0 \neq v \in H^1(D)} \frac{\Re a(u, v)}{\|u\|_V\|v\|_V}.
  \end{align*}
  This completes the proof.
\end{proof}

We now derive the approximation error of the global multiscale method.
\begin{lemma}
  Let $u$ and $u_{H}$ solve \eqref{equ:definition-weak-form} and \eqref{equ:discrete-weak-from-msfem}, respectively, and the resolution condition \cref{assump:mesh-size-wavenumber} hold. Then, the approximation error of the global multiscale method satisfies
  \begin{equation*}
    \|u - u_H\|_{L^2(D)} \leq C H^{2+s} (1 + \kappa H) \|f\|_{H^s(D)},
  \end{equation*}
  and
  \begin{equation*}
    \|u - u_H\|_V \leq CH^{s}\|f\|_{H^s(D)} + (1 + \kappa)H^{2+s} \|f\|_{H^s(D)}.
  \end{equation*}
  where $C$ is independent of $\kappa$.
  \label{lem:error-analysis-msfem}
\end{lemma}
\begin{proof}
  Denote $e_H := u - u_H \in W$ the error function. Let $v = v_H$ in \eqref{equ:definition-weak-form}, and we have the classical Galerkin orthogonality $a(e_H, v_H) = 0$ for all $v_H \in \Psi_H$.
  Then, the boundary error satisfies
  \begin{align}
    \kappa \sqrt{n_1} \|e_H\|^2_{L^2(\partial D)} & \leq \Im a(e_H, e_H) = \Im f(e_H) = (f - \mathcal{I}_Hf, e_H) \nonumber \\
    & \leq \|f - \mathcal{I}_Hf\|_{L^2(D)} \|e_H\|_{L^2(D)}.
    \label{equ:boundary-error-estimate}
  \end{align}
  And
  \begin{equation}
    \|\nabla e_H\|^2_{L^2(D)} - \kappa^2(n e_H, e_H) \leq \kappa \sqrt{n_2} \|e_H\|^2_{L^2(\partial D)} \lesssim \|f - \mathcal{I}_Hf\|_{L^2(D)} \|e_H\|_{L^2(D)}.
    \label{equ:real-part-error-estimate}
  \end{equation}
  We next estimate the $L^2$-error. Using the Nitsche's duality argument, consider the adjoint problem
  \begin{equation*}
    \begin{aligned}
      &-\Delta w - \kappa^2n w = \rho \quad \mathbf{x}\in D, \\
      &\nabla w \cdot \nu = i\kappa \sqrt{n} w \quad \mathbf{x}\in \partial D.
    \end{aligned}
  \end{equation*}
  We then have the weak form
  \begin{equation}
    a(v, w) = (v, \rho), \quad \forall v \in H^1(D).
    \label{equ:weak-form-adjoint-problem}
  \end{equation}
  It can be proved that $w$ satisfies $|w|_{H^2(D)} \lesssim \|\rho\|_{L^2(D)}$.
  Then, let $v = \rho = e_H$ and it has
  \begin{align*}
    & \|e_H\|^2_{L^2(D)} = a(e_H, w) = a(e_H, w - \mathcal{I}_Hw) \\
    = & (\nabla e_H, \nabla (w - \mathcal{I}_Hw)) - \kappa^2(n(\mathbf{x})e_H, w - \mathcal{I}_Hw) - i\kappa \langle \sqrt{n(\mathbf{x})}e_H, w - \mathcal{I}_Hw \rangle.
  \end{align*}
  Owing to the approximation error \eqref{equ:approximation-interpolation-operator}, we have
  \begin{equation*}
    \|w - \mathcal{I}_Hw \|_{L^2(\partial D)} \leq CH^{\frac{3}{2}} |w|_{H^2(D)}.
  \end{equation*}
  We then arrive at
  \begin{align*}
    \|e_H\|^2_{L^2(D)} \lesssim & \|\nabla e_H\|_{L^2(D)} H |w|_{H^2(D)} + \kappa^2n_2\|e_H\|_{L^2(D)} H^2|w|_{H^2(D)} \\
    &+ \kappa \sqrt{n_2} \|e_H\|_{L^2(\partial D)} H^{\frac 32} |w|_{H^2(D)}.
  \end{align*}
  If \cref{assump:mesh-size-wavenumber} holds, we also have
  \begin{align*}
    \|e_H\|^2_{L^2(D)} &\lesssim H^2\|\nabla e_H\|^2_{L^2(D)} + \kappa^2 H^3\|e_H\|^2_{L^2(\partial D)}.
  \end{align*}
  Substituting \eqref{equ:boundary-error-estimate} and \eqref{equ:real-part-error-estimate} yields
  \begin{align*}
    \|e_H\|^2_{L^2(D)} \lesssim & H^2 \left(\kappa^2 n_2 \|e_H\|^2_{L^2(D)} + \|f - \mathcal{I}_Hf\|_{L^2(D)} \|e_H\|_{L^2(D)} \right) \\
    & + \kappa H^3 \|f - \mathcal{I}_Hf\|_{L^2(D)} \|e_H\|_{L^2(D)},
  \end{align*}
  which indicates $\|e_H\|_{L^2(D)} \lesssim H^2(1 + \kappa H) \|f - \mathcal{I}_Hf\|_{L^2(D)}$.
  Since $$\|e_H\|^2_V \leq C(1 + \kappa^2) \|e_H\|^2_{L^2(D)} + \|f - \mathcal{I}_H f\|^2_{L^2(D)},$$ we arrive at
  \begin{equation*}
    \|e_H\|_V \lesssim \|f - \mathcal{I}_H f\|_{L^2(D)} + (1 + \kappa)\|e_H\|_{L^2(D)}.
  \end{equation*}
  This finishes the proof.
\end{proof}

\subsubsection{Localized multiscale basis}
Define the patches $\{D_{\ell}\}$ associated with $\mathbf{x}_i \in \mathcal{N}_H$
\begin{align*}
  &D_0(\mathbf{x}_i) := \mathrm{supp}\{\phi_i\} = \cup\{K\in\mathcal{T}_H\;|\; \mathbf{x}_i\in K\}, \\
  &D_{\ell} := \cup\{K\in\mathcal{T}_H\;|\;K\cap\overline{D_{\ell-1}}\neq \emptyset\},\quad \ell = 1,2,\cdots.
\end{align*}
The localized basis $\Phi_{H, \ell} = \mathrm{span}\{\Phi_j\}$ is then constructed by solving the optimal problems as follows.
\begin{equation}
  \begin{aligned}
    &\Phi_j = \argmin_{\phi\in (H^1(D_{\ell}; \mathds{R}))^2} \hat{a}_{\ell}(\phi, \phi), \\
    &\mathrm{s.t.} \quad \left(\Phi, \phi_k^H \right)_{L^2(D_{\ell})} = \alpha_j\delta_{jk}(1, 1)^T, \; \forall 1 \leq k \leq N_H,
  \end{aligned}
  \label{equ:optimal-problem-localized}
\end{equation}
where
\begin{equation*}
  \hat{a}_{\ell}(\phi, \phi) = -\kappa\int_{\partial D \cap \overline{D_{\ell}}} \sqrt{n(\mathbf{x})}(\phi_R^2 - \phi_I^2) \mathrm{d}\mathbf{s} + \int_{D \cap \overline{D_{\ell}}} |\nabla \phi|^2 \mathrm{d}\mathbf{x} - \kappa^2\int_{D \cap \overline{D_{\ell}}} n(\mathbf{x}) |\phi|^2 \mathrm{d}\mathbf{x}.
\end{equation*}

Owing to the exponential decay of the basis function over the full domain, the localized multiscale method has the approximation error related to the oversampling size $\ell$ as the following lemma.
\begin{lemma}
  If the resolution condition of \cref{assump:mesh-size-wavenumber} is satisfied, there exist constant $\beta < 1$ independent of $H$ and $\kappa$ and $\ell \gtrsim \log \kappa$ such that
  \begin{equation}
    \|u - u_{H,\ell}\|_V \lesssim \left(H^{s} + \kappa H^{2+s}\right) \|f\|_{H^s(D)} + c_{\ell}\beta^{\ell} \|f\|_{L^2(D)}.
  \end{equation}
  \label{lem:approximation-error-localized-MsFEM}
\end{lemma}
Its proof follows the same procedures as in \cite{fda66b15-cbd8-3cbb-a160-13398a1c4325,Wu2022}, hence omit it here.

\section{Regularity with respect to the parametric variables}
\label{sec:regularity-wrt-variables}
We need a function space defined with respect to the stochastic parameter $\boldsymbol{\omega}$, a weighted Sobolev space $\mathcal{W}_{s, \boldsymbol{\gamma}}(\Omega; H^1(D))$ with the norm of $F \in \mathcal{W}_{s, \boldsymbol{\gamma}}(\Omega; L^2(D))$
\begin{multline}
  \|F\|_{\mathcal{W}_{s, \boldsymbol{\gamma}}(\Omega; L^2(D))} := \\
  \left(\sum_{\mathfrak{u} \subseteq \{1:s\} } \gamma_{\mathfrak{u}}^{-1} \int_{[-\frac 12, \frac 12]^{|\mathfrak{u}|}} \left(\int_{[-\frac 12, \frac 12]^{s - |\mathfrak{u}|}} \frac{\partial^{|\mathfrak{u}|} F }{\partial \boldsymbol{\omega}_{\mathfrak{u}}}(\boldsymbol{\omega}_{\mathfrak{u}}; \boldsymbol{\omega}_{\{1:s\}\setminus \mathfrak{u}}; 0)  \mathrm{d}\boldsymbol{\omega}_{\{1:s\}\setminus \mathfrak{u}} \right)^2 \mathrm{d}\mathfrak{u} \right)^{1/2},
  \label{equ:definition-norm-W-for-F}
\end{multline}
where $\{1:s\}$ represents the shorthand notation of indices set $\{1, 2, \cdots, s\}$, $\frac{\partial^{|\mathfrak{u}|} F }{\partial \boldsymbol{\omega}_{\mathfrak{u}}}$ denotes the mixed derivative with respect to $\omega_j$ with $j \in \mathfrak{u}$, and $\boldsymbol{\omega}_{\{1:s\}\setminus \mathfrak{u}}$ denotes the variable $\omega_j$ with $j \notin \mathfrak{u}$.

The fundamental quantity in uncertainty quantification is
\begin{equation*}
  \int_{\Omega} F(\boldsymbol{\omega}) \mathrm{d} \boldsymbol{\omega}, \quad \text{ with } \quad F(\boldsymbol{\omega}) = G(u(\cdot, \boldsymbol{\omega})),
\end{equation*}
where $G \in L^2(D)$ is a functional.
Then the corresponding norm is
\begin{multline*}
  \|u\|_{\mathcal{W}_{s, \boldsymbol{\gamma}}(\Omega; H^1(D))} := \\
  \left(\sum_{\mathfrak{v} \subseteq \{1:s\} } \gamma_{\mathfrak{v}}^{-1} \int_{[-\frac 12, \frac 12]^{|\mathfrak{v}|}} \left\|\int_{[-\frac 12, \frac 12]^{s - |\mathfrak{v}|}} \frac{\partial^{|\mathfrak{v}|} u }{\partial \boldsymbol{\omega}_{\mathfrak{v}}}(\cdot, (\boldsymbol{\omega}_{\mathfrak{v}}; \boldsymbol{\omega}_{\{1:s\}\setminus \mathfrak{v}}; 0)) \mathrm{d}\boldsymbol{\omega}_{\{1:s\}\setminus \mathfrak{v}} \right\|^2_V \mathrm{d}\mathfrak{v} \right)^{1/2}.
\end{multline*}
For $s \geq 1$ and $u$ satisfying $\|u\|_{\mathcal{W}_{s, \boldsymbol{\gamma}}(\Omega; H^1(D))} < \infty$, since
\begin{equation*}
  \frac{\partial^{|\mathfrak{u}|}F}{\partial \boldsymbol{\omega}_{\mathfrak{u}}} = G\left( \frac{\partial^{|\mathfrak{u}|}u}{\partial \boldsymbol{\omega}_{\mathfrak{u}}}(\cdot, \boldsymbol{\omega}) \right),
\end{equation*}
using \eqref{equ:definition-norm-W-for-F} yields
\begin{equation}
  \|F\|_{\mathcal{W}_{s, \boldsymbol{\gamma}}(\Omega; L^2(D))} \leq \|G(\cdot)\|_{L^2(D)} \|u\|_{\mathcal{W}_{s, \boldsymbol{\gamma}}(\Omega; H^1(D))} < \infty.
\end{equation}

We now estimate the regularity of the solution $u(\mathbf{x}, \boldsymbol{\omega})$ with respect to the random stochastic variables.
\begin{theorem}
  Under the assumption on the bounds of random refractive index \eqref{equ:bounds-for-n}, for any multi-index $\boldsymbol{\nu}$ with $|\boldsymbol{\nu}| \ge 0$, we have
  \begin{gather*}
    \kappa \|\partial^{\boldsymbol{\nu}}u(\mathbf{x}, \boldsymbol{\omega}) \|_{L^2(D)} \leq C (1 + \kappa^{|\boldsymbol{\nu}|}) |\boldsymbol{\nu}|! \prod_j \|\psi_j\|_{\infty}^{m_j} \|f\|_{L^2(D)}, \\
    \|\partial^{\boldsymbol{\nu}}u(\mathbf{x}, \boldsymbol{\omega}) \|_{V} \leq C (1 + \kappa^{|\boldsymbol{\nu}|}) |\boldsymbol{\nu}|! \prod_j \|\psi_j\|_{\infty}^{m_j} \|f\|_{L^2(D)},
  \end{gather*}
  where $\sum_j m_j = |\boldsymbol{\nu}|$, $C$ is independent of $\kappa$ but depends on the domain, $\mu$, $n_j$ and $\beta_j (j = 1,2)$.
  \label{thm:regularity-u-wrt-stochastic}
\end{theorem}
\begin{proof}
  \Cref{prop:stability-estimate} gives the validation for $|\boldsymbol{\nu}| = 0$.
  When $|\boldsymbol{\nu}| > 0$, we use the Leibniz product rule
  \begin{equation*}
    \partial^{\boldsymbol{\nu}}(PQ) = \sum_{\mathbf{m} \leq \boldsymbol{\nu}} {\boldsymbol{\nu} \choose \mathbf{m}} \partial^{\mathbf{m}}P \partial^{\boldsymbol{\nu} - \mathbf{m}}Q.
  \end{equation*}
  To the Helmoltz equation with the random refractive index \eqref{equ:refractive-index-infinite-series}, we have
  \begin{equation}
    -\Delta \partial^{\boldsymbol{\nu}} u - \kappa^2 n(\mathbf{x}, \boldsymbol{\omega})\partial^{\boldsymbol{\nu}} u - \kappa^2 \sum_{j=1}^{\infty} \psi_j(\mathbf{x}) \partial^{\boldsymbol{\nu}-\mathbf{e}_j} u = 0.
    \label{equ:Helmoltz-equation-derivatives-random}
  \end{equation}
  Meanwhile, since
  \begin{equation*}
    \partial^{\mathbf{m}}\sqrt{n(\mathbf{x}, \boldsymbol{\omega})} = (-1)^{|\mathbf{m}|} \frac{(2|\mathbf{m}| - 1)!!}{2^{|\mathbf{m}|}} \left(n(\mathbf{x}, \boldsymbol{\omega})\right) ^{\frac 12 - |\mathbf{m}|} \prod_{j}(\psi_j)^{m_j},
  \end{equation*}
  the corresponding boundary condition of \eqref{equ:Helmoltz-equation-derivatives-random} is
  \begin{equation}
    \nabla \partial^{\boldsymbol{\nu}} u \cdot \nu - i\kappa\sum_{\mathbf{m} \leq \boldsymbol{\nu}} {\boldsymbol{\nu} \choose \mathbf{m}}(-1)^{|\mathbf{m}|} \frac{(2|\mathbf{m}| - 1)!!}{2^{|\mathbf{m}|}} \left(n(\mathbf{x}, \boldsymbol{\omega})\right) ^{\frac 12 - |\mathbf{m}|} \prod_{j}(\psi_j)^{m_j} \partial^{\boldsymbol{\nu} - \mathbf{m}}u = 0.
    \label{equ:Helmoltz-boundary-derivatives-random}
  \end{equation}
  We therefore define the weak form of \eqref{equ:Helmoltz-equation-derivatives-random}-\eqref{equ:Helmoltz-boundary-derivatives-random} as
  \begin{equation}
    \begin{aligned}
      -\int_{\partial D} i\kappa\sum_{\mathbf{m} \leq \boldsymbol{\nu}} {\boldsymbol{\nu} \choose \mathbf{m}}(-1)^{|\mathbf{m}|} \frac{(2|\mathbf{m}| - 1)!!}{2^{|\mathbf{m}|}} \left(n(\mathbf{x}, \boldsymbol{\omega})\right) ^{\frac 12 - |\mathbf{m}|} \prod_{j}(\psi_j)^{m_j} \partial^{\boldsymbol{\nu} - \mathbf{m}}u v \mathrm{d}s\\
    + \int_D \nabla \partial^{\boldsymbol{\nu}} u \cdot \nabla v \mathrm{d}\mathbf{x} - \kappa^2\int_D n(\mathbf{x}, \boldsymbol{\omega})\partial^{\boldsymbol{\nu}} u v \mathrm{d}\mathbf{x} - \kappa^2 \int_D \sum_{j=1}^{\infty} \psi_j(\mathbf{x}) \partial^{\boldsymbol{\nu}-\mathbf{e}_j} u v \mathrm{d}\mathbf{x} = 0.
    \end{aligned}
    \label{equ:random-derivative-3}
  \end{equation}
  Without generality, we first let $\boldsymbol{\nu} = \boldsymbol{e}_j$. Repeating the derivation as in the proof of \cref{prop:stability-estimate} yields
  \begin{equation*}
    \kappa \|\partial^{\mathbf{e}_j}u\|_{L^2(D)} \leq C\|\psi_j\|_{\infty}(1 + \kappa) \|f\|_{L^2(D)},
  \end{equation*}
  and
  \begin{equation*}
    \|\partial^{\mathbf{e}_j}u\|_{L^2(\partial D)} \leq C \|\psi_j\|_{\infty}\|f\|_{L^2(D)}, \quad \|\partial^{\mathbf{e}_j}u\|_V \leq C\|\psi_j\|_{\infty}(1 + \kappa) \|f\|_{L^2(D)}.
  \end{equation*}

  We next derive a general result using the mathematical induction method. Assume for all $|\boldsymbol{\nu}| < n$, there holds
  \begin{gather*}
    \|\partial^{\boldsymbol{\nu}} u\|_{L^2(\partial D)} \leq C (1 + \kappa^{|\boldsymbol{\nu}|-1}) |\boldsymbol{\nu}|! \prod_{j} b_j^{\nu_j} \|f\|_{L^2(D)},\\
    \kappa\|\partial^{\boldsymbol{\nu}} u\|_{L^2(D)} \leq C (1 + \kappa^{|\boldsymbol{\nu}|}) |\boldsymbol{\nu}|! \prod_{j} b_j^{\nu_j}\|f\|_{L^2(D)},
  \end{gather*}
  where we denote $b_j = \|\psi_j\|_{\infty}$.
  Let $v = \partial^{\boldsymbol{\nu}} \bar{u}$ in \eqref{equ:random-derivative-3} and denote
  \begin{multline}
    R(\partial^{\boldsymbol{\nu}} u) = \kappa^2 \int_D \sum_{j=1}^{\infty} \psi_j(\mathbf{x}) \partial^{\boldsymbol{\nu}-\mathbf{e}_j} u \partial^{\boldsymbol{\nu}} \bar{u} \mathrm{d}\mathbf{x}\\
      + \int_{\partial D} i\kappa\sum_{\mathbf{0} \neq \mathbf{m} \leq \boldsymbol{\nu}} {\boldsymbol{\nu} \choose \mathbf{m}}(-1)^{|\mathbf{m}|} \frac{(2|\mathbf{m}| - 1)!!}{2^{|\mathbf{m}|}} \left(n(\mathbf{x}, \boldsymbol{\omega})\right) ^{\frac 12 - |\mathbf{m}|} \prod_{j}(\psi_j)^{m_j} \partial^{\boldsymbol{\nu} - \mathbf{m}}u \partial^{\boldsymbol{\nu}} \bar{u} \mathrm{d}s.
  \end{multline}
  We then have
  \begin{gather*}
      \|\nabla \partial^{\boldsymbol{\nu}} u\|^2_{L^2(D)} - \kappa^2(n(\mathbf{x}, \boldsymbol{\omega}) \partial^{\boldsymbol{\nu}} u, \partial^{\boldsymbol{\nu}} u) = \Re R(\partial^{\boldsymbol{\nu}} u), \\
    \kappa(\sqrt{n(\mathbf{x}, \boldsymbol{\omega})}\partial^{\boldsymbol{\nu}} u, \partial^{\boldsymbol{\nu}} u)_{L^2(\partial D)} = \Im R(\partial^{\boldsymbol{\nu}} u).
  \end{gather*}
  Note that
  \begin{align*}
    |R(\partial^{\boldsymbol{\nu}} u)| \leq & \kappa^2 \sum_{j=1}^{\infty} b_j \|\partial^{\boldsymbol{\nu}-\mathbf{e}_j} u\|_{L^2(D)}\|\partial^{\boldsymbol{\nu}} u\|_{L^2(D)} \\
      &+ \kappa\sum_{k = 1}^{|\boldsymbol{\nu}|} \sum_{ \substack{ \mathbf{m} \leq \boldsymbol{\nu} \\|\mathbf{m}| = k} } {\boldsymbol{\nu} \choose \mathbf{m}} \frac{(2|\mathbf{m}| - 1)!!}{2^{|\mathbf{m}|}} \frac{\sqrt{n_1}}{n_1^{|\mathbf{m}|}} \prod_{j}b_j^{m_j} \|\partial^{\boldsymbol{\nu} - \mathbf{m}}u\|_{L^2(\partial D)} \|\partial^{\boldsymbol{\nu}} u\|_{L^2(\partial D)}.
  \end{align*}
  We have
  \begin{equation*}
    \|\partial^{\boldsymbol{\nu}} u\|^2_{L^2(\partial D)} \leq \frac{1}{\kappa\sqrt{n_1}} |R(\partial^{\boldsymbol{\nu}} u)|
  \end{equation*}
  and
  \begin{align*}
    & \sum_{k=1}^{|\boldsymbol{\nu}|} \sum_{ \substack{ \mathbf{m} \leq \boldsymbol{\nu} \\|\mathbf{m}| = k} } {\boldsymbol{\nu} \choose \mathbf{m}} \frac{(2|\mathbf{m}| - 1)!!}{(2n_1)^{|\mathbf{m}|}} \prod_{j} b_j^{m_j} \|\partial^{\boldsymbol{\nu} - \mathbf{m}}u\|_{L^2(\partial D)} \|\partial^{\boldsymbol{\nu}} u\|_{L^2(\partial D)} \\
    \leq & \frac{1}{2}\bigg\{ \sum_{k=1}^{|\boldsymbol{\nu}|} {|\boldsymbol{\nu}| \choose k} \frac{(2k - 1)!!}{(2n_1)^{k}} (1 + \kappa^{|\boldsymbol{\nu}| - k - 1}) (|\boldsymbol{\nu}| - k)! \prod_{j} b_j^{\nu_j} \|f\|_{L^2(D)} \bigg\}^2 + \frac{1}{2}\|\partial^{\boldsymbol{\nu}} u\|^2_{L^2(\partial D)}
  \end{align*}
  Here we used the identities $|\boldsymbol{\nu} - \mathbf{m}| = |\boldsymbol{\nu}| - |\mathbf{m}|$ and
  \begin{equation*}
    \sum_{ \substack{ \mathbf{m} \leq \boldsymbol{\nu}, |\mathbf{m}| = k} } {\boldsymbol{\nu} \choose \mathbf{m}}  = {|\boldsymbol{\nu}| \choose k}.
  \end{equation*}
  We then obtain
  \begin{align*}
    \|\partial^{\boldsymbol{\nu}} u\|^2_{L^2(\partial D)} \leq & \frac{2\kappa}{\sqrt{n_1}}\sum_{j=1}^{\infty} b_j \|\partial^{\boldsymbol{\nu}-\mathbf{e}_j} u\|_{L^2(D)}\|\partial^{\boldsymbol{\nu}} u\|_{L^2(D)} \\
    & + \bigg\{ \sum_{k=1}^{|\boldsymbol{\nu}|} {|\boldsymbol{\nu}| \choose k} \frac{(2k - 1)!!}{(2n_1)^{k}} (1 + \kappa^{|\boldsymbol{\nu}| - k - 1}) (|\boldsymbol{\nu}| - k)! \prod_{j} b_j^{\nu_j} \|f\|_{L^2(D)} \bigg\}^2
  \end{align*}
  and
  \begin{align*}
    \|\nabla \partial^{\boldsymbol{\nu}} u\|^2_{L^2(D)} \leq n_2\kappa^2 \|\partial^{\boldsymbol{\nu}} u\|^2_{L^2(D)} + |R(\partial^{\boldsymbol{\nu}} u)|.
  \end{align*}
  Besides, we also have
  \begin{align}
    |R(\partial^{\boldsymbol{\nu}} u)| \leq & 2\kappa^2 \sum_{j=1}^{\infty} b_j \|\partial^{\boldsymbol{\nu}-\mathbf{e}_j} u\|_{L^2(D)}\|\partial^{\boldsymbol{\nu}} u\|_{L^2(D)} + \nonumber\\
    & + \kappa\sqrt{n_1} \bigg\{ \sum_{k=1}^{|\boldsymbol{\nu}|} {|\boldsymbol{\nu}| \choose k} \frac{(2k - 1)!!}{(2n_1)^{k}} (1 + \kappa^{|\boldsymbol{\nu}| - k - 1}) (|\boldsymbol{\nu}| - k)! \prod_{j} b_j^{\nu_j} \|f\|_{L^2(D)} \bigg\}^2.
    \label{equ:RHS-estimate1}
  \end{align}
  Then, let $v = \alpha \cdot \nabla \partial^{\boldsymbol{\nu}} u$ in \eqref{equ:random-derivative-3} and we obtain
  \begin{align*}
    & \frac 12\int_{\partial D} (\alpha \cdot \nu) |\nabla \partial^{\boldsymbol{\nu}}u|^2 \mathrm{d}s + \frac{\kappa^2}{2} \int_D \nabla n(\mathbf{x}, \boldsymbol{\omega}) \cdot \alpha |\partial^{\boldsymbol{\nu}}u|^2 \mathrm{d}\mathbf{x} + \frac{\kappa^2}{2} \int_D n(\mathbf{x}, \boldsymbol{\omega}) |\partial^{\boldsymbol{\nu}}u|^2 \mathrm{d}\mathbf{x} \\
    = & \Re \int_{\partial D} i\kappa\sum_{\mathbf{m} \leq \boldsymbol{\nu}} {\boldsymbol{\nu} \choose \mathbf{m}}(-1)^{|\mathbf{m}|} \frac{(2|\mathbf{m}| - 1)!!}{2^{|\mathbf{m}|}} \left(n(\mathbf{x}, \boldsymbol{\omega})\right) ^{\frac 12 - |\mathbf{m}|} \prod_{j}(\psi_j)^{m_j} \partial^{\boldsymbol{\nu} - \mathbf{m}}u \alpha \cdot \nabla \partial^{\boldsymbol{\nu}}\bar{u} \mathrm{d}s \\
    & + \left( \frac{d}{2} - 1\right)\|\nabla \partial^{\boldsymbol{\nu}}u\|^2_{L^2(D)} - \left( \frac{d}{2} - \frac 12 \right) \kappa^2 (n(\mathbf{x}, \boldsymbol{\omega})\partial^{\boldsymbol{\nu}}u, \partial^{\boldsymbol{\nu}}u) \\
    & + \frac{\kappa^2}{2} \int_{\partial D} n(\mathbf{x}, \boldsymbol{\omega}) (\alpha \cdot \nu) |\partial^{\boldsymbol{\nu}}u|^2 \mathrm{d}s + \kappa^2 \Re \int_D \sum_{j=1}^{\infty} \psi_j \partial^{\boldsymbol{\nu} - \mathbf{e}_j}u \alpha \cdot \nabla\partial^{\boldsymbol{\nu}}\bar{u} \mathrm{d}\mathbf{x}.
  \end{align*}
  This infers
  \begin{align*}
    & \frac{\mu\kappa^2}{2} \|\partial^{\boldsymbol{\nu}} u\|^2_{L^2(D)} + \frac{\beta_1}{2}\|\nabla\partial^{\boldsymbol{\nu}} u\|^2_{L^2(\partial D)} \\
    \leq & M\sqrt{n_2}\kappa\left( \frac{\eta_0}{2}\|\partial^{\boldsymbol{\nu}} u\|^2_{L^2(\partial D)} + \frac{1}{2\eta_0}\|\nabla \partial^{\boldsymbol{\nu}} u\|^2_{L^2(\partial D)} \right) \\
    & + M\kappa \sum_{k=1}^{|\boldsymbol{\nu}|} \sum_{ \substack{\mathbf{m} \leq \boldsymbol{\nu} \\ |\mathbf{m}| = k }} {\boldsymbol{\nu} \choose \mathbf{m}} \frac{(2k - 1)!!}{2^{k}} \frac{\sqrt{n_1}}{n_1^{k}} \prod_{j}b_j^{m_j} \|\partial^{\boldsymbol{\nu} - \mathbf{m}} u\|_{L^2(\partial D)} \|\nabla \partial^{\boldsymbol{\nu}} u\|_{L^2(\partial D)} \\
    & + \frac 12 |R(\partial^{\boldsymbol{\nu}} u)| + \frac{\beta_2n_2\kappa^2}{2} \|\partial^{\boldsymbol{\nu}} u\|^2_{L^2(\partial D)} + M\kappa^2 \sum_{j=1}^{\infty} b_j \|\partial^{\boldsymbol{\nu}-\mathbf{e}_j} u\|_{L^2(D)} \|\nabla \partial^{\boldsymbol{\nu}} u\|_{L^2(D)}.
  \end{align*}
  Choose $\eta_0 = \frac{2M\sqrt{n_2}\kappa}{\beta_1}$, and use the inequality
  \begin{align*}
    & \sum_{k=1}^{|\boldsymbol{\nu}|} \sum_{ \substack{\mathbf{m} \leq \boldsymbol{\nu}, |\mathbf{m}| = k }} M\kappa {\boldsymbol{\nu} \choose \mathbf{m}} \frac{(2k - 1)!!}{2^{k}} \frac{\sqrt{n_1}}{n_1^{k}} \prod_{j}b_j^{m_j} \|\partial^{\boldsymbol{\nu} - \mathbf{m}} u\|_{L^2(\partial D)} \|\nabla \partial^{\boldsymbol{\nu}} u\|_{L^2(\partial D)} \\
    \leq
    & \frac{CM^2\kappa^2n_1\eta_1}{2} \bigg[ \sum_{k=1}^{|\boldsymbol{\nu}|} {|\boldsymbol{\nu}| \choose k} \frac{(2k - 1)!!}{(2n_1)^{k}} \prod_{j}b_j^{\nu_j} (1 + \kappa^{|\boldsymbol{\nu}|-k-1}) (|\boldsymbol{\nu}| - k)! \|f\|_{L^2(D)} \bigg]^2 + \frac{1}{2\eta_1} \|\nabla \partial^{\boldsymbol{\nu}} u\|^2_{L^2(\partial D)}.
  \end{align*}
  We further choose $\eta_1 = \frac{2}{\beta_1}$ and arrive at
  \begin{align*}
    &\frac{\mu\kappa^2}{2} \|\partial^{\boldsymbol{\nu}} u\|^2_{L^2(D)} \leq \frac{M^2n_2\kappa^2}{\beta_1} \|\partial^{\boldsymbol{\nu}} u\|^2_{L^2(\partial D)} + \frac 12|R(\partial^{\boldsymbol{\nu}} u)| + \frac{\beta_2n_2\kappa^2}{2} \|\partial^{\boldsymbol{\nu}} u\|^2_{L^2(\partial D)} \\
    \leq & \left( \frac{M^2}{\beta_1} + \frac{\beta_2}{2} \right)n_2\kappa^2 \|\partial^{\boldsymbol{\nu}} u\|^2_{L^2(\partial D)} + \left( \frac 12 + \frac{M\kappa^2}{2\eta_2} \right) |R(\partial^{\boldsymbol{\nu}} u)| + \frac{Mn_2\kappa^4}{2\eta_2} \|\partial^{\boldsymbol{\nu}} u\|^2_{L^2(D)} \\
    & + \frac{CM^2\kappa^2n_1}{\beta_1} \bigg\{ \sum_{k=1}^{|\boldsymbol{\nu}|} {|\boldsymbol{\nu}| \choose k} \frac{(2k - 1)!!}{(2n_1)^{k}} \prod_{j}b_j^{\nu_j} (1 + \kappa^{|\boldsymbol{\nu}|-k-1}) (|\boldsymbol{\nu}| - k)! \|f\|_{L^2(D)} \bigg\}^2 \\
    & + \frac{M\kappa^2\eta_2}{2} \bigg( \sum_{j=1}^{\infty} b_j \|\partial^{\boldsymbol{\nu}-\mathbf{e}_j} u\|_{L^2(D)} \bigg)^2.
  \end{align*}
  Using the inequalities
  \begin{gather*}
    2\sum_{j=1}^{\infty} b_j \|\partial^{\boldsymbol{\nu}-\mathbf{e}_j} u\|_{L^2(D)}\|\partial^{\boldsymbol{\nu}} u\|_{L^2(D)} \leq \eta_3 \bigg( \sum_{j=1}^{\infty} b_j \|\partial^{\boldsymbol{\nu}-\mathbf{e}_j} u\|_{L^2(D)} \bigg)^2 + \frac{1}{\eta_3} \|\partial^{\boldsymbol{\nu}} u\|^2_{L^2(D)} \\
    2\sum_{j=1}^{\infty} b_j \|\partial^{\boldsymbol{\nu}-\mathbf{e}_j} u\|_{L^2(D)}\|\partial^{\boldsymbol{\nu}} u\|_{L^2(D)} \leq \eta_4 \bigg( \sum_{j=1}^{\infty} b_j \|\partial^{\boldsymbol{\nu}-\mathbf{e}_j} u\|_{L^2(D)} \bigg)^2 + \frac{1}{\eta_4} \|\partial^{\boldsymbol{\nu}} u\|^2_{L^2(D)}.
  \end{gather*}
  We finally get
  \begin{align*}
    & \frac{\mu\kappa^2}{2} \|\partial^{\boldsymbol{\nu}} u\|^2_{L^2(D)} \\
    \leq & \left[ \left( \frac{M^2}{\beta_1} + \frac{\beta_2}{2} \right) \frac{n_2\kappa^3}{\sqrt{n_1}\eta_3} + \left(\frac 12 + \frac{M\kappa^2}{2\eta_2}\right) \frac{\kappa^2}{\eta_4} + \frac{Mn_2\kappa^4}{2\eta_2} \right] \|\partial^{\boldsymbol{\nu}} u\|^2_{L^2(D)} + \\
    & \left[ \left( \frac{M^2}{\beta_1} + \frac{\beta_2}{2} \right) \frac{n_2\kappa^3\eta_3}{\sqrt{n_1}} + \left(\frac 12 + \frac{M\kappa^2}{2\eta_2}\right) \eta_4\kappa^2 + \frac{M\kappa^2\eta_2}{2} \right] \bigg( \sum_{j=1}^{\infty} b_j \|\partial^{\boldsymbol{\nu}-\mathbf{e}_j} u\|_{L^2(D)} \bigg)^2 \\
    & + \left[ \left( \frac{M^2}{\beta_1} + \frac{\beta_2}{2} \right) n_2\kappa^2 + \left(\frac 12 + \frac{M\kappa^2}{2\eta_2}\right) \sqrt{n_1}\kappa + \frac{CM^2\kappa^2n_1}{\beta_1} \right] \times \\
    &\bigg\{ \sum_{k=1}^{|\boldsymbol{\nu}|} {|\boldsymbol{\nu}| \choose k} \frac{(2k - 1)!!}{(2n_1)^{k}} \prod_{j}b_j^{\nu_j} (1 + \kappa^{|\boldsymbol{\nu}|-k-1}) (|\boldsymbol{\nu}| - k)! \|f\|_{L^2(D)} \bigg\}^2.
  \end{align*}
  Choose
  \begin{equation*}
    \eta_2 = \frac{Mn_2\kappa^2}{2C_2\mu}, \quad \eta_3 = \left( \frac{M^2}{\beta_1} + \frac{\beta_2}{2} \right)\frac{n_2\kappa}{C_3\sqrt{n_1}\mu}, \quad \eta_4 = \frac{1}{2C_4\mu} + \frac{C_2}{2n_2C_4},
  \end{equation*}
  and let $C_2 + C_3 + C_4 < \frac 12$. There exists a constant $C$ independent of $\kappa$ such that
  \begin{multline*}
    \kappa^2 \|\partial^{\boldsymbol{\nu}} u\|^2_{L^2(D)} \leq C\kappa^2( 1 + \kappa^2 ) \bigg( \sum_{j=1}^{\infty} b_j \|\partial^{\boldsymbol{\nu}-\mathbf{e}_j} u\|_{L^2(D)} \bigg)^2 \\
    + C\kappa^2(1 + \kappa^2) \bigg\{ \sum_{k=1}^{|\boldsymbol{\nu}|} {|\boldsymbol{\nu}| \choose k} \frac{(2k - 1)!!}{(2n_1)^{k}} \prod_{j}b_j^{\nu_j} (1 + \kappa^{|\boldsymbol{\nu}|-k-1}) (|\boldsymbol{\nu}| - k)! \|f\|_{L^2(D)} \bigg\}^2.
  \end{multline*}
  Owing to
  \begin{align*}
    & \kappa(1 + \kappa)\sum_{k=1}^{m} {m \choose k} \frac{(2k - 1)!!}{(2n_1)^{k}} (1 + \kappa^{m-k-1}) (m - k)! \\
    \leq & C\kappa(1 + \kappa) (1 + \kappa^{m-2}) \sum_{k=1}^{m} \frac{m!}{k! (m-k)!} \frac{(2k - 1)!!}{(2n_1)^{k}} (m - k)! \\
    \leq & C(1 + \kappa^{m})m! \sum_{k=1}^{m} \frac{(2k - 1)!!}{(2n_1)^{k}k!}.
  \end{align*}
  We then have
  \begin{equation}
    \kappa \|\partial^{\boldsymbol{\nu}} u\|_{L^2(D)} \leq C( 1 + \kappa^{|\boldsymbol{\nu}|}) |\boldsymbol{\nu}|! \prod b_j^{\nu_j} \|f\|_{L^2(D)},
  \end{equation}
  as well as
  \begin{equation}
    \|\partial^{\boldsymbol{\nu}} u\|_{L^2(\partial D)} \leq C( 1 + \kappa^{|\boldsymbol{\nu}|-1}) |\boldsymbol{\nu}|! \prod b_j^{\nu_j} \|f\|_{L^2(D)}.
  \end{equation}
  These complete the proof.
\end{proof}

\section{Dimension truncation and qMC integrals}
\label{sec:dimension-truncation}
In order to simulate the solution of the Helmholtz equation with the series form of the refractive index, we are of course necessary to truncate the infinite series expansion and to control the resulting error. Recall the expansion series
$$n(\mathbf{x}, \boldsymbol{\omega}) = n_0(\mathbf{x}) + \sum_{j=1}^{\infty} \omega_j \psi_j(\mathbf{x}).$$
The functions $\psi_j(\mathbf{x})$ can arise from either the eigensystem of a covariance operator or other suitable function system in $L^2(D)$.
For the convergence analysis, we impose some assumptions on the functions $\{\psi_j$\} as follows.
\begin{assumption}
  For the series \eqref{equ:refractive-index-infinite-series}, assume:
  \begin{itemize}
    \item We have $0 < n_0(\mathbf{x}) \in L^{\infty}(D)$ and $\sum_{j=1}^{\infty}\|\psi_j\|_{\infty} < \infty$;
    \item The sequence $\psi_j$ is ordered so that $\|\psi_1\|_{\infty} \geq \|\psi_2\|_{\infty} \geq \cdots$;
    \item There exits $p \in (0, 1)$ such that $\sum_{j=1}^{\infty}\|\psi_j\|^{p} < \infty$.
  \end{itemize}
  \label{assump:for-KL-series}
\end{assumption}

The approximation of the random refractive index parameterized by the dimensionally truncated KL expansion is
\begin{equation}
  n^s(\mathbf{x}, \boldsymbol{\omega}) = n_0(\mathbf{x}) + \sum_{j=1}^s \psi_s(\mathbf{x}) \omega_s, \quad \text{for some } s \in \mathds{N}.
  \label{equ:truncated-random-refractive-index}
\end{equation}
Then, define $u^s(\cdot, \boldsymbol{\omega}) \in H^1(D)$ to be the solution of the dimensionally truncated boundary value problem
\begin{equation}
  a^s(\boldsymbol{\omega}; u^s, v) = F(v), \quad \text{for all } v \in H^1(D),
  \label{equ:weak-form-with-randomness}
\end{equation}
where
\begin{equation}
  a^s(\boldsymbol{\omega}; u^s, v) = \int_D \nabla u^s \nabla v \mathrm{d}\mathbf{x} - \kappa^2\int_D n^s(\mathbf{x}, \boldsymbol{\omega}) u^s v \mathrm{d}\mathbf{x} - i\kappa \int_{\partial D} \sqrt{n^s(\mathbf{x}, \boldsymbol{\omega})} u^s v \mathrm{d}s.
\end{equation}

\begin{lemma}\label{lem:gap-u-us}
  Let $u$ and $u^s$ be the solution of \eqref{equ:definition-weak-form} and \eqref{equ:weak-form-with-randomness}, respectively. The dependence of the gap between $u$ and $u^s$ for the truncated random refractive index satisfies
  \begin{gather*}
    \|u - u^s\|_{L^2(D)} \leq C(1 + \kappa^{-1})\|f\|_{L^2(D)} \|n^s(\mathbf{x}, \boldsymbol{\omega}) - n(\mathbf{x}, \boldsymbol{\omega})\|_{\infty},\\
    \|u - u^s\|_V \leq C(1 + \kappa)\|f\|_{L^2(D)} \|n^s(\mathbf{x}, \boldsymbol{\omega}) - n(\mathbf{x}, \boldsymbol{\omega})\|_{\infty}.
  \end{gather*}
\end{lemma}
\begin{proof}
  Since $f$ is independent of the stochastic variable, we have $a(\boldsymbol{\omega}; u, v) - a^s(\boldsymbol{\omega}; u^s, v) = 0$ for all $v \in H^1(D)$, i.e.,
  \begin{equation*}
    \int_D \nabla(u^s - u) \cdot \nabla \bar{v} \mathrm{d}\mathbf{x} - \kappa^2 \int_{D} n^s(\mathbf{x}, \boldsymbol{\omega}) (u^s - u) \bar{v} \mathrm{d}\mathbf{x} - i\kappa \int_{\partial D} \sqrt{n^s(\mathbf{x}, \boldsymbol{\omega})} (u^s - u) \bar{v} \mathrm{d}s = R(v), \label{equ:weak-form-gap-u-us}
  \end{equation*}
  where
  \begin{equation*}
    R(v) = \kappa^2\int_D (n^s(\mathbf{x}, \boldsymbol{\omega}) - n(\mathbf{x}, \boldsymbol{\omega}))u \bar{v} \mathrm{d}\mathbf{x} + i\kappa\int_{\partial D} \left(\sqrt{n(\mathbf{x}, \boldsymbol{\omega})} - \sqrt{n^s(\mathbf{x}, \boldsymbol{\omega})}\right)u \bar{v} \mathrm{d}s.
  \end{equation*}
  Choosing $v = {u}^s - {u}$ yields
  \begin{gather*}
    \|\nabla (u - u^s)\|^2_{L^2(D)} - \kappa^2(n^s(\mathbf{x}, \boldsymbol{\omega})(u - u^s), u-u^s)_{L^2(D)} = \Re R(v), \\
    \kappa (\sqrt{n^s(\mathbf{x}, \boldsymbol{\omega})}(u - u^s), u-u^s)_{L^2(\partial D)} = \Im R(v).
  \end{gather*}
  Since
  \begin{align*}
    |R(v)| \leq & \kappa^2 \| n^s(\mathbf{x}, \boldsymbol{\omega}) - n(\mathbf{x}, \boldsymbol{\omega}) \|_{\infty} \|u\|_{L^2(D)} \|u-u^s\|_{L^2(D)} \\
    & + \frac{\kappa}{2\sqrt{n_2}} \|n(\mathbf{x}, \boldsymbol{\omega}) - n^s(\mathbf{x}, \boldsymbol{\omega})\|_{\infty} \|u\|_{L^2(\partial D)} \|u-u^s\|_{L^2(\partial D)},
  \end{align*}
  in which we use the inequality
  \begin{equation*}
    \|\sqrt{n(\mathbf{x}, \boldsymbol{\omega})} - \sqrt{n^s(\mathbf{x}, \boldsymbol{\omega})}\|_{\infty} = \left\| \frac{n(\mathbf{x}, \boldsymbol{\omega}) - n^s(\mathbf{x}, \boldsymbol{\omega})}{\sqrt{n(\mathbf{x}, \boldsymbol{\omega})} + \sqrt{n^s(\mathbf{x}, \boldsymbol{\omega})}} \right\|_{\infty} \leq \frac{1}{2\sqrt{n_2}} \|n(\mathbf{x}, \boldsymbol{\omega}) - n^s(\mathbf{x}, \boldsymbol{\omega})\|_{\infty}
  \end{equation*}
  with $0 < n_1 \leq \|n(\mathbf{x}, \boldsymbol{\omega})\|_{\infty}, \|n^s(\mathbf{x}, \boldsymbol{\omega})\|_{\infty} \leq n_2$ for all $\boldsymbol{\omega} \in \Omega$. We then obtain
  \begin{align*}
    \|u-u^s\|^2_{L^2(\partial D)} \leq & \frac{2\kappa}{\sqrt{n_1}} \| n^s(\mathbf{x}, \boldsymbol{\omega}) - n(\mathbf{x}, \boldsymbol{\omega}) \|_{\infty} \|u\|_{L^2(D)} \|u-u^s\|_{L^2(D)} \\
    & + \frac{1}{4n_1n_2} \| n^s(\mathbf{x}, \boldsymbol{\omega}) - n(\mathbf{x}, \boldsymbol{\omega}) \|^2_{\infty} \|u\|^2_{L^2(\partial D)},
  \end{align*}
  and
  \begin{align*}
    |R(v)| \leq & \left(\kappa^2 + \frac{\kappa}{\sqrt{n_1}\epsilon_2}\right) \| n^s(\mathbf{x}, \boldsymbol{\omega}) - n(\mathbf{x}, \boldsymbol{\omega}) \|_{\infty} \|u\|_{L^2(D)} \|u-u^s\|_{L^2(D)} \\
    & + \left( \frac{\epsilon_2\kappa^2}{8n_2} + \frac{1}{8n_1n_2\epsilon_2} \right) \| n^s(\mathbf{x}, \boldsymbol{\omega}) - n(\mathbf{x}, \boldsymbol{\omega}) \|^2_{\infty} \|u\|^2_{L^2(\partial D)}.
  \end{align*}
  Next, let $v = \alpha \cdot \nabla ({u}^s - {u})$ and we obtain
  \begin{align*}
    & \frac{\mu\kappa^2}{2} \|u^s - u\|^2_{L^2(D)} \\
    \leq & \frac{\epsilon_3}{2}\left\{ \frac{1 + \epsilon_1}{2}\left( \kappa^2 + \frac{\kappa}{\sqrt{n_1}\epsilon_2} \right) + \frac{2M^2n_2\kappa^3}{\sqrt{n_1}\beta_1} \right\}^2 \|n^s - n\|^2_{\infty}\|u\|^2_{L^2(D)} + \frac{1}{2\epsilon_3}\|u^s - u\|^2_{L^2(D)}  \\
    & + \left\{ \left[\frac{(1+\epsilon_1)\epsilon_2\kappa^2}{16n_2} + \frac{1+\epsilon_1}{16n_1n_2\epsilon_2} + \frac{M^2\kappa^2}{4n_1\beta_1} \right] \|u\|^2_{L^2(\partial D)} + \frac{M^2\kappa^4}{2\epsilon_1} \|u\|^2_{L^2(D)} \right\} \|n^s - n\|^2_{\infty} \\
    & + \frac{\epsilon_1n_2\kappa^2}{2} \|u^s - u\|^2_{L^2(D)} + \frac{M^2\kappa^2}{\beta_1} \|\sqrt{n} - \sqrt{n^s}\|^2_{\infty} \|u\|^2_{L^2(\partial D)}.
  \end{align*}
  Choose $\epsilon_1 = 2C_1\mu / n_2$ and $\epsilon_3 = 1 / (2C_3\mu\kappa^2)$ and let $C_1 + C_3 < \frac 12$. Then, there exists a positive constant $C$ such that
  \begin{align*}
    \kappa^2 \|u^s - u\|^2_{L^2(D)} \leq & C (1 + \kappa)^2 \kappa^2 \|n^s - n\|^2_{\infty} \|u\|^2_{L^2(D)} \\
    & + C(1 + \kappa^2) \|u\|^2_{L^2(\partial D)} \|n^s - n\|^2_{\infty} + \kappa^4 \|u\|^2_{L^2(D)} \|n^s - n\|^2_{\infty},
  \end{align*}
  in which we choose $\epsilon_2 = \mathcal{O}(1)$ and $C$ depends on the domain, $n_j (j = 1,2)$, $\beta_1$ but is independent of $\kappa$. According to \cref{prop:stability-estimate}, we obtain
  \begin{equation*}
    \kappa^2 \|u^s - u\|^2_{L^2(D)} \leq C (1 + \kappa)^2 \|f\|^2_{L^2(D)} \|n^s - n\|^2_{\infty},
  \end{equation*}
  and
  \begin{equation*}
    \|u^s - u\|^2_{V} \leq C (1 + \kappa)^2 \|f\|^2_{L^2(D)} \|n^s - n\|^2_{\infty}.
  \end{equation*}
  which finishes the proof of this lemma.
\end{proof}

Given $s \in \mathds{N}$ and $\boldsymbol{\omega} \in \Omega$, we truncate the series \eqref{equ:refractive-index-infinite-series} by the setup $\omega_j = 0$ for $j > s$. Under the condition \eqref{equ:bounds-for-n} and \Cref{assump:for-KL-series}, it holds \cite{doi:10.1137/110845537}
\begin{equation}
  \sum_{j \geq s+1} \|\psi_j\|_{\infty} \leq \min \Big(\frac{1}{1/p-1}, 1\Big) \bigg(\sum_{j = 1}^{\infty}\|\psi_j\|_{\infty}^p \bigg)^{1/p} s^{1-1/p}.
\end{equation}
Consequently, we have the truncation error as follows.

\begin{theorem}
  Under the condition \eqref{equ:bounds-for-n} and \Cref{assump:for-KL-series}, for all $f \in L^2(D)$, and for every $\boldsymbol{\omega} \in \Omega$, $s \in \mathds{N}$ and $p \in (0, 1)$, the truncated solution $u^s(\mathbf{x}, \boldsymbol{\omega}_s)$ satisfies
  \begin{equation}
    \kappa\|u^s - u\|_{L^2(D)}, \|u^s - u\|_V \leq C(1 + \kappa) \|f\|_{L^2(D)} s^{1-1/p},
    \label{equ:truncation-error-solution}
  \end{equation}
  where the constant $C > 0$ is independent of $\kappa$, $s$. Furthermore, for every $G \in L^2(D)$, we have
  \begin{equation}
    |I(G(u)) - I_s(G(u))| \leq {C}(1 + \kappa) \|f\|_{L^2(D)} \|G(\cdot)\|_{L^2(D)} s^{1-2/p}.
    \label{equ:truncation-error-integral-solution}
  \end{equation}
  \label{thm:approximation-error-truncation}
\end{theorem}

In the above \Cref{thm:approximation-error-truncation},  \eqref{equ:truncation-error-solution} can be calculated directly using \cref{lem:gap-u-us}. As for the proof of \eqref{equ:truncation-error-integral-solution}, it can be derived by following the proof of Theorem 5.1 in \cite{ganesh2021quasi}.

For a real-valued function $F$ depending on infinite variables $\boldsymbol{\omega} = (\omega_1, \omega_2,
\cdots)$, we apply the qMC method by ``anchoring to 0'' all components beyond the dimension $s$. In particular, the integral reads
\begin{equation*}
  I(F) = \lim_{s\rightarrow \infty} I_s(F) = \lim_{s\rightarrow \infty} \int_{[-\frac 12, \frac12]^s}F(\boldsymbol{\omega}_s; 0) \mathrm{d}\omega_1 \cdots \mathrm{d}\omega_s.
\end{equation*}
A $N$-point qMC approximation to the integral with a randomly shifted lattice rule is
\begin{equation*}
  Q_{s, N}(F; \boldsymbol{\Delta}) = \frac 1N\sum_{j=1}^N F \Big( \{\boldsymbol{\omega}^j_s + \boldsymbol{\Delta}\} -  \mathbf{\frac{1}{2}} \Big),
\end{equation*}
where $\boldsymbol{\omega}^1_s, \cdots, \boldsymbol{\omega}^N_s$ are deterministic lattice cubature points that carefully chosen from $[0, 1]^s$, and $\boldsymbol{\Delta}$ is a random shift drawn from a uniform distribution on $[0, 1]^s$. The lattice points are given by $\boldsymbol{\omega}^j = \{\frac{j\mathbf{z}}{N}\}$ for $j = 1, \cdots, N$ with $\mathbf{z} \in \mathds{Z}^s$ known as the generating vector and determined by the quality of lattice rule. Then, given $N$ a prime power and the weights $\gamma_{\mathfrak{u}}$ as input, a generating vector $\mathbf{z}$ can be obtained by the component-by-component (CBC) construction to achieve the root-mean-square error (see, e.g., \cite{Dick_Kuo_Sloan_2013})
\begin{equation*}
  \sqrt{\mathds{E}[ |I_s(F) - Q_{s, N}(F; \cdot)|^2 ]} \leq \bigg( \frac{2}{N} \sum_{\emptyset \neq \mathfrak{u} \subseteq \{1:s\}} \gamma^{\lambda}_{\mathfrak{u}} [\varrho(\lambda)]^{|\mathfrak{u}|} \bigg)^{1/(2\lambda)} \|F\|_{\mathcal{W}_{s, \boldsymbol{\gamma}}}
\end{equation*}
for all $\lambda \in (\frac 12, 1]$, where $\varrho(\lambda) = \frac{2\zeta(2\lambda)}{(2\pi^2)^{\lambda}}$ with $\zeta$ being the Riemann zeta function.

To the Helmholtz problem, owing to the linearity and boundedness of $G$, we have
\begin{equation}
  \left| \frac{\partial^{|\mathfrak{u}|}}{\partial \boldsymbol{\omega}_{\mathfrak{u}}} G(u^s(\cdot, \boldsymbol{\omega}))  \right| = \left| G\left( \frac{\partial^{|\mathfrak{u}|}}{\partial \boldsymbol{\omega}_{\mathfrak{u}}} u^s( \cdot, \boldsymbol{\omega} ) \right) \right| \leq \|G\|_{L^2(D)} \left\| \frac{\partial^{|\mathfrak{u}|}}{\partial \boldsymbol{\omega}_{\mathfrak{u}}} u^s(\cdot, \boldsymbol{\omega}) \right\|_V.
\end{equation}
Applying \Cref{thm:regularity-u-wrt-stochastic}, we obtain a similar result like in
\cite{ganesh2021quasi} as follows.
\begin{lemma}
  Let $u^s$ be the solution of the parametric form \eqref{equ:weak-form-with-randomness}, $G \in L^2(D)$ and $f \in L^2(D)$. A generating vector can be constructed for a randomly shifted lattice rule such that
  \begin{equation}
    \sqrt{\mathds{E}[ |I_s(G(u^s(\mathbf{x}, \boldsymbol{\omega}))) - Q_{s, N}(G(u^s(\mathbf{x}, \boldsymbol{\omega}); \cdot)|^2 ]} \leq \frac{C_{s,\gamma}(\lambda)}{N^{1/(2\lambda)}} \|f\|_{L^2(D)} \|G\|_{L^2(D)},
  \end{equation}
  where
  \begin{equation*}
    C_{s,\gamma}(\lambda) = 2^{1/(2\lambda)} \Bigg( \sum_{\mathfrak{u} \subseteq \{1:s\}} \Big( |\mathfrak{u}|!\prod_{j \in \mathfrak{u}} (1+\kappa)\|\psi_j\|_{W^{1,\infty}} [\varrho(\lambda)]^{1/(2\lambda)} \Big)^{\frac{2\lambda}{1+\lambda}} \Bigg)^{\frac{1 + \lambda}{2\lambda}},
  \end{equation*}
  with
  \begin{equation*}
    \lambda = \left\{\begin{aligned}
      & \frac{1}{2 - 2\delta} \quad \text{ for some } \delta \in (0, \frac 12) \quad & p \in (0, \frac 23], \\
      & \frac{p}{2-p} \quad & p \in (\frac 23, 1).
    \end{aligned}\right.
  \end{equation*}
  \label{lem:approximation-error-qMC}
\end{lemma}
So, in order to obtain the above error estimate, the weights $\gamma_{\mathfrak{u}}$ should be chosen carefully and they have been given in \cite{ganesh2021quasi}. This leads to a convergence rate $\mathcal{O}(N^{-\min(\frac{1}{p}-\frac 12, 1-\delta)})$ with an implied constant independent of $s$.

\section{The overall numerical approach}
\label{sec:overall-method-analysis}
\subsection{Main algorithm}
The formal algorithm is then outlined as \Cref{alg:MsFEM-POD-Helmholtz}.

\begin{algorithm}[ht]
  \caption{The qMC multiscale method for random Helmholtz equation.}
  \label{alg:MsFEM-POD-Helmholtz}
  \begin{algorithmic}[1]
    \REQUIRE QMC samples $\{\boldsymbol{\omega}_j\}_{j=1}^N$, coarse mesh $\mathcal{T}_H$, fine mesh $\mathcal{T}_h$, and the oversampling size $\ell$.
        \FOR{each $j \in \{1 : N\}$}
    \STATE Solve optimal problems \eqref{equ:optimal-problem-localized} and generate basis set $\Psi_{H,\ell} = span\{\phi_k\}_{k=1}^{N_H}$;
    \STATE Find $u_{H,\ell}(\boldsymbol{\omega}^j) \in \Psi_{H,\ell}$ such that
    \begin{equation}
      a(\boldsymbol{\omega}^j; u^s_{H, \ell}, v) = F(v), \quad\forall v \in \Psi_{H,\ell}.
      \label{equ:fully-discretized-weak-form}
    \end{equation}
    \ENDFOR
    \STATE Compute the expectation $\mathds{E}(u_{H,\ell})$.
  \end{algorithmic}
\end{algorithm}

\subsection{Convergence analysis of the fully discrete approximation}

The error estimate of the algorithm is a combined form that consists of the truncation error, the spatial approximation error and the qMC integral error. Specifically, we now estimate
\begin{equation*}
  \sqrt{\mathds{E}[|I(G(u(\mathbf{x}, \boldsymbol{\omega}))) - Q_{s,N}(G(u_{H,\ell}^{s}(\mathbf{x}, \boldsymbol{\omega})))|^2]},
\end{equation*}
where $u_H^{s}(\mathbf{x}, \boldsymbol{\omega})$ is the solution of \eqref{equ:fully-discretized-weak-form}.
We split this error into three parts as
\begin{align*}
  &|I(G(u(\mathbf{x}, \boldsymbol{\omega}))) - Q_{s,N}(G(u_{H,\ell}^{s}(\mathbf{x}, \boldsymbol{\omega})))| \\
  \leq & |I(G(u(\mathbf{x}, \boldsymbol{\omega}))) - I_s(G(u(\mathbf{x}, \boldsymbol{\omega})))| + |I_s(G(u^s(\mathbf{x}, \boldsymbol{\omega}))) - Q_{s,N}(G(u^{s}(\mathbf{x}, \boldsymbol{\omega})))| \\
  & + |Q_{s,N}(G(u^{s}(\mathbf{x}, \boldsymbol{\omega}))) - Q_{s,N}(G(u_{H,\ell}^{s}(\mathbf{x}, \boldsymbol{\omega})))|,
\end{align*}
in which it holds true that $I_s(G(u(\mathbf{x}, \boldsymbol{\omega}))) = I_s(G(u^s(\mathbf{x}, \boldsymbol{\omega})))$.
Combining the estimates from the above subsections with \Cref{lem:approximation-error-localized-MsFEM}, \Cref{thm:approximation-error-truncation} and \Cref{lem:approximation-error-qMC} yields the main conclusion of this paper.

\begin{theorem}\label{thm:main-result}
  Let \cref{assump:mesh-size-wavenumber}, \cref{assump:for-KL-series}, and the bounds on random refractive index \eqref{equ:bounds-for-n} hold. For every $\boldsymbol{\omega} \in \Omega$, let $u(\mathbf{x}, \boldsymbol{\omega})$ be the solution of \eqref{equ:definition-weak-form} and $u^{s}_{H, \ell}$ be the solution of \eqref{equ:fully-discretized-weak-form}. Then for all $f \in H^2(D)$ and the linear functional $G \in L^2(D)$, a generating vector can be constructed for a randomly shifted lattice rule such that the approximation error satisfies
  \begin{equation*}
      \sqrt{\mathds{E}[|I(G(u )) - Q_{s,N}(G(u_{H,\ell}^{s} ))|^2]} \leq C\left( H^2 + \beta^{\ell} + N^{-\alpha} + (1 + \kappa) s^{1-2/p} \right),
  \end{equation*}
  where $\alpha = \min(\frac 1p - \frac 12, 1-\delta)$, the constant $C$ mainly depends on $\|f\|_{H^2(D)}$, $\|G\|_{L^2(D)}$ and the oversampling size $\ell$, but is independent of $\kappa$, $s$, $N$ and $H$.
\end{theorem}

\section{Numerical experiments}
\label{sec:numerics}
Next, we verify the error estimate given in \Cref{thm:main-result}, excluding the localization error.


\subsection{The multiscale method for the deterministic problems}
The deterministic problems mean that the refractive index is independent of random variables.

\begin{example}
  Consider the 1D problem over the interval $[0, 1]$, and fix $ k = 2^6$ and the fine mesh size $h = 2^{-14}$. Below, four parameters' configurations are employed: (1) $n(x) = 1$ and a discontinuous right-hand side $f(x)$ as
\begin{equation}
  f(x) = \left\{\begin{aligned}
    &2\sqrt{2}, \quad &x \in [\frac{3}{16}, \frac{5}{16}] \cup [\frac{11}{16}, \frac{13}{16}], \\
    &0, \quad &\text{elsewhere};
  \end{aligned}\right.
  \label{equ:1d-discontinuous-rhs}
\end{equation}
(2) $n(x) = (1 + \sin(\pi x))^2$ and $f(x)$ is given by \eqref{equ:1d-discontinuous-rhs}; (3) $n(x) = (1 + \sin(\pi x))^2$ and $f(x) = 1$; (4) $n(x) = \exp(-2\cos(\pi x))$ and $f(x) = 1$.
\begin{figure}[htbp]
  \centering
  \subfloat[$n(x) = 1$, $f$ given by \eqref{equ:1d-discontinuous-rhs}.]{\includegraphics[width=0.45\linewidth]{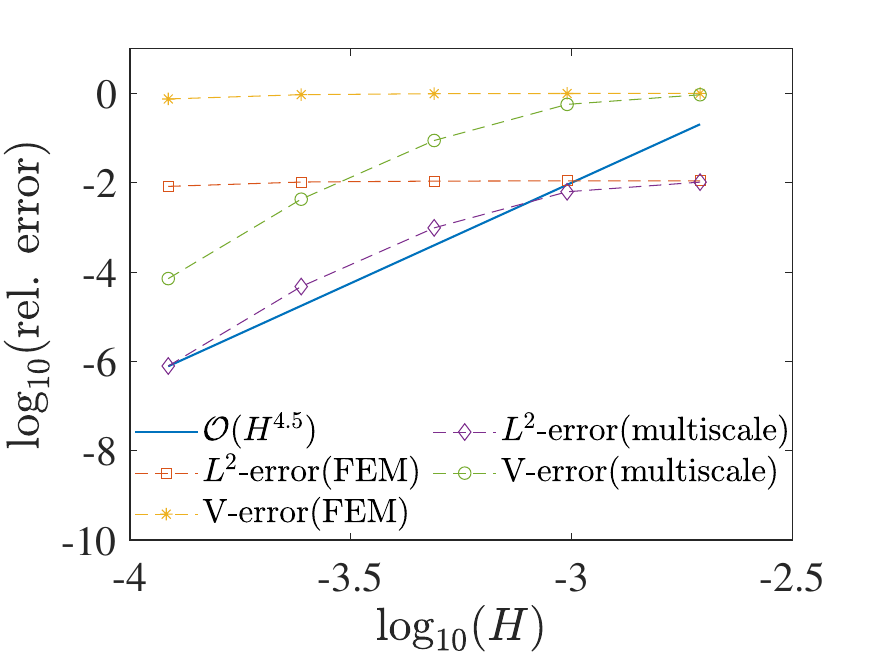}}
  \subfloat[$n(x) = (1 + \sin(\pi x))^2$, $f$ given by \eqref{equ:1d-discontinuous-rhs}.]{\includegraphics[width=0.45\linewidth]{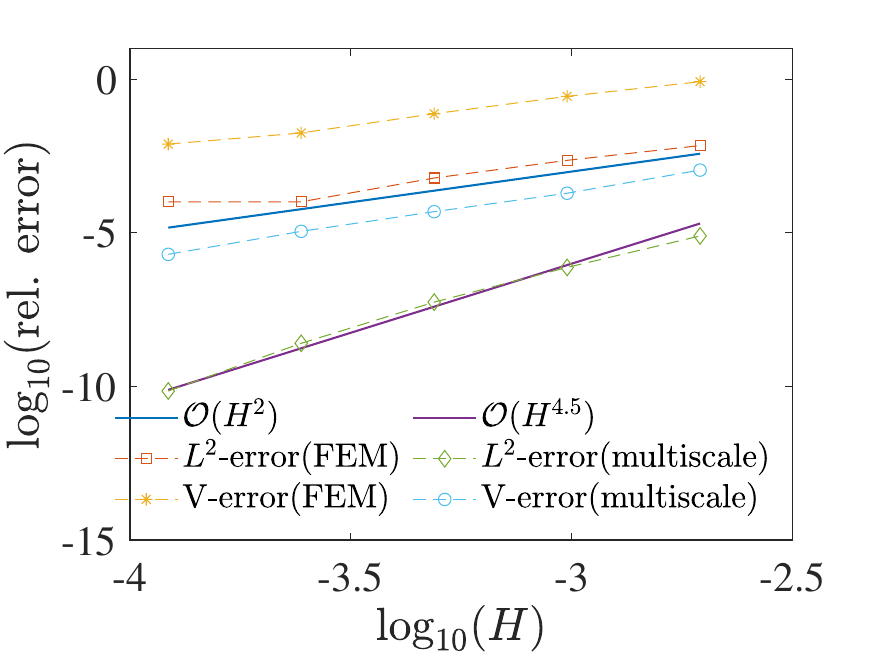}}
  \\
  \subfloat[$n(x) = (1 + \sin(\pi x))^2$, $ f(x) = 1$.]{\includegraphics[width=0.45\linewidth]{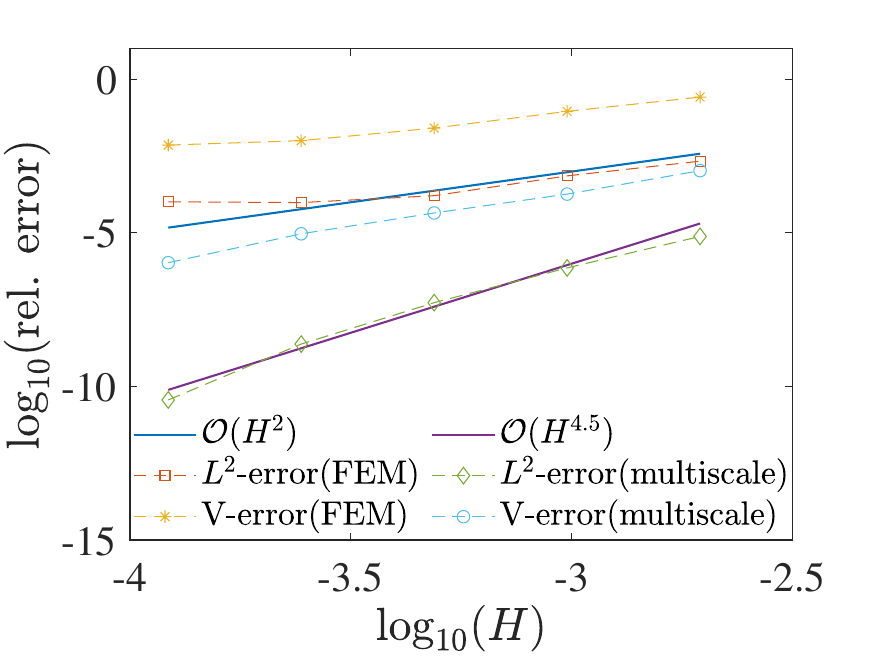}}
  \subfloat[$n(x) = \exp(-2\cos(\pi x))$, $ f(x) = 1$.]{\includegraphics[width=0.45\linewidth]{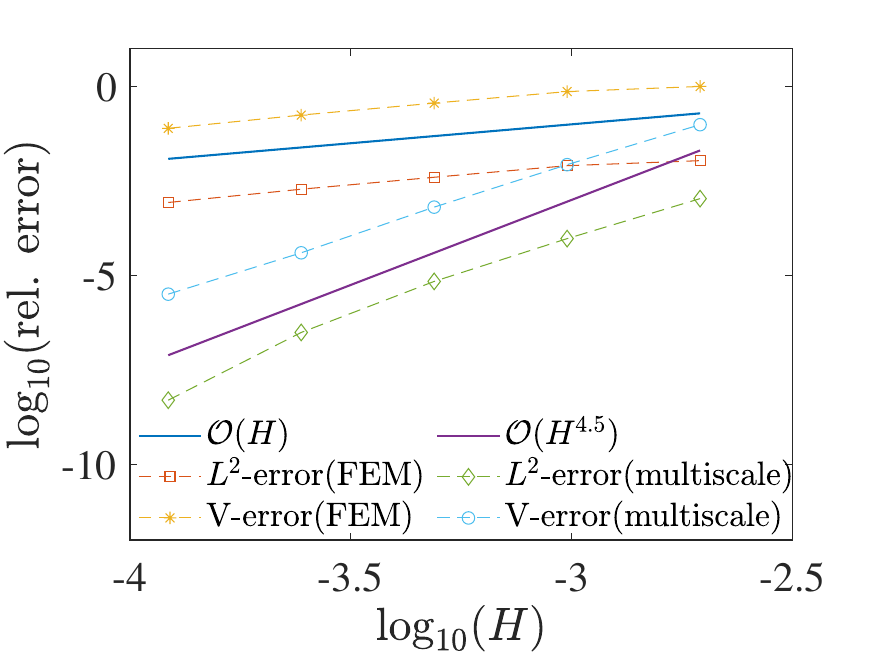}}
  \caption{The convergence rates of the FEM and multiscale method for the 1D Helmholtz equation with different configurations of $n(x)$ and $f(x)$.}
  \label{fig:1d-convergence-msfem}
\end{figure}

As shown in \Cref{fig:1d-convergence-msfem}, both the $L^2$-error and $V$-error of the proposed multiscale method can reach the superconvergence rate $\mathcal{O}(H^4)$ and $\mathcal{O}(H^2)$. The multiscale method of course has a better performance than the FEM for all configurations of $n(x)$ and $f(x)$. We note that the $V$-error of multiscale also shows that the convergence rate can reach $\mathcal{O}(H^4)$.
\end{example}

\begin{example}
  We consider the 2D problem over the domain $D = [0, 1]^2$ and fix $f(\mathbf{x}) = 1$. The reference solution is calculated by the FEM with $h = 2^{-9}$. Meanwhile, we fix $\kappa = 2^3$ here and choose different forms of the deterministic refractive index: (i) $n(x, y) = (1.0 + \sin(\pi x)\sin(\pi y))^2$ and (ii) $n(x, y) = \exp(-2\sin(\pi x)\cos(\pi y))$. We vary the mesh size $H$ and the convergence rates of both $L^2$-error and $V$-error (depicted in \Cref{fig:2d-convergence_msfem}) strictly confirm the error analysis in \Cref{lem:error-analysis-msfem}.
\begin{figure}[htbp]
  \centering
  \subfloat[$n(x, y) = (1 + \sin(\pi x)\sin(\pi y))^2$.]{\includegraphics[width=0.45\linewidth]{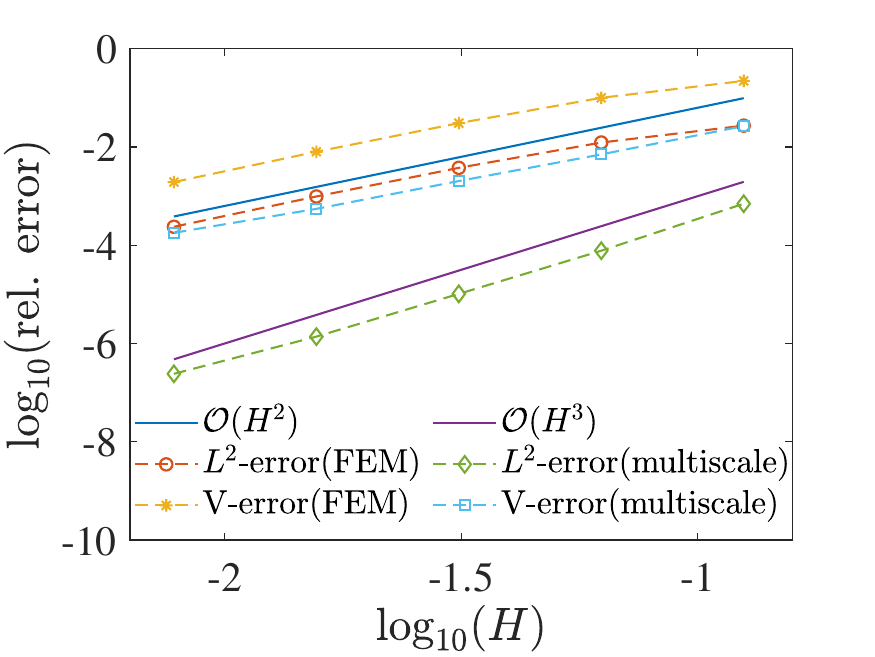}}
  \subfloat[$n(x, y) = \exp(-2\sin(\pi x)\cos(\pi y))$.]{\includegraphics[width=0.45\linewidth]{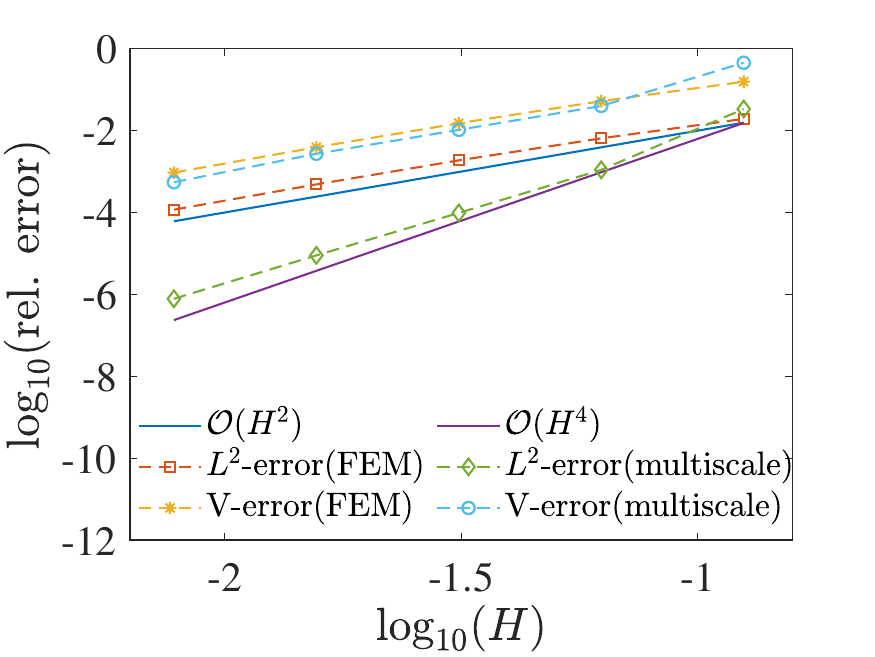}}
  \caption{The convergence rates of the multiscale method for the 2D Helmholtz equation with different $n(x, y)$.}
  \label{fig:2d-convergence_msfem}
\end{figure}
\end{example}

\begin{example}[Localized multiscale method]
  We choose this example from \cite{doi:10.1137/21M1465950}, and intend to demonstrate the effectiveness of the boundary corrector provided in this paper. Consider the heterogeneous Helmholtz equation
  \begin{equation*}
    \begin{aligned}
      &-\nabla \cdot (A\nabla u) - \kappa^2 u = f, &\text{ in } D,\\
    &\nabla u \cdot \nu - i\kappa u = 0, &\text{ on } \partial D,
    \end{aligned}
  \end{equation*}
  where $A$ is piece-wise constant with respect to a quadrilateral background mesh with mesh size $\mathcal{O}(\epsilon)$ and $0 < \epsilon \ll 1$. On each quadrilateral, $A$ takes either the value $\epsilon^2$ or 1. Over the domain $D = [0, 1]^2$, we employ the periodic scatterer in the domain $(0.25, 0.75)^2$ with $D_{\epsilon} = (0.25, 0.75)^2 \cup_{j \in \mathds{Z}^2} \epsilon (j + (0.25, 0.75)^2)$ and let
  \begin{equation*}
    f(x, y) = \left\{\begin{aligned}
      &10^4\exp\left( \frac{-1}{1-\frac{(x-z_1)^2 + (y-z_2)^2}{0.05^2}} \right), &(x-z_1)^2 + (y-z_2)^2 < 0.05^2, &\\
      &0, &\text{elsewhere}. &
    \end{aligned}\right.
  \end{equation*}

  The reference solution is computed using the standard FEM with the mesh size $h = 2^{-9}$. The computation is finished on the coarse mesh with the mesh size $H = 2^{-6}$. The relative error is computed with the norm $\| \cdot \|_{A, \mathcal{V}}^2 = \|A\nabla \cdot\|^2 + \kappa^2\|\cdot\|$. Without the boundary correction, the relative error is $2.23\cdot10^{-2}$, while after the boundary correction being employed, the relative error reduces to $4.23\cdot 10^{-4}$ with $\ell = 10$, which closes to the Super-LOD error $3.3\cdot 10^{-4}$ reported in \cite{doi:10.1137/21M1465950}. The error distributions in space of the localized multiscale method with and without the boundary corrector are compared as in \Cref{fig:comparsion_msfem_boundary_correction}.
  \begin{figure}[htbp]
    \centering
    \includegraphics[width=0.45\linewidth]{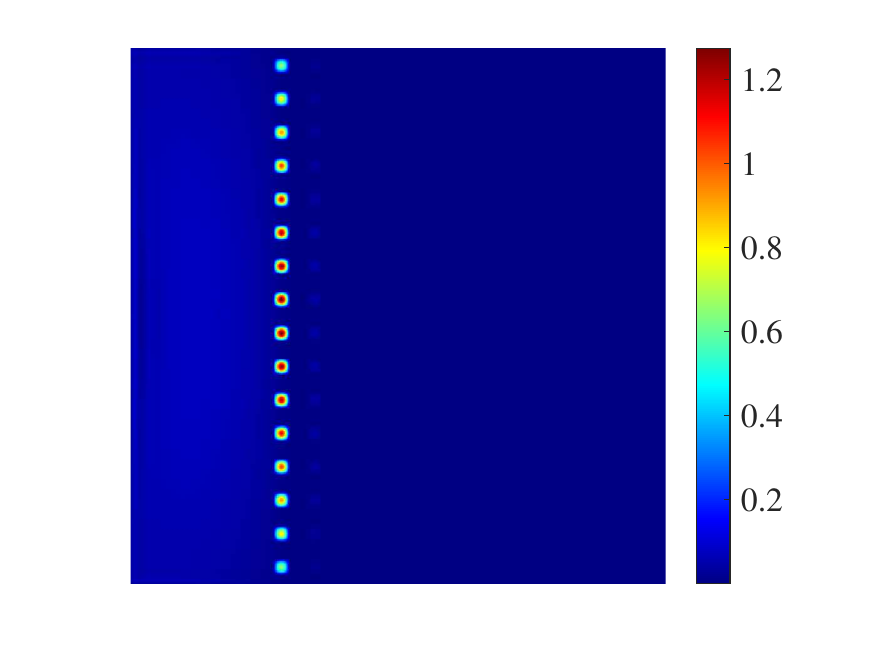}
    \includegraphics[width=0.45\linewidth]{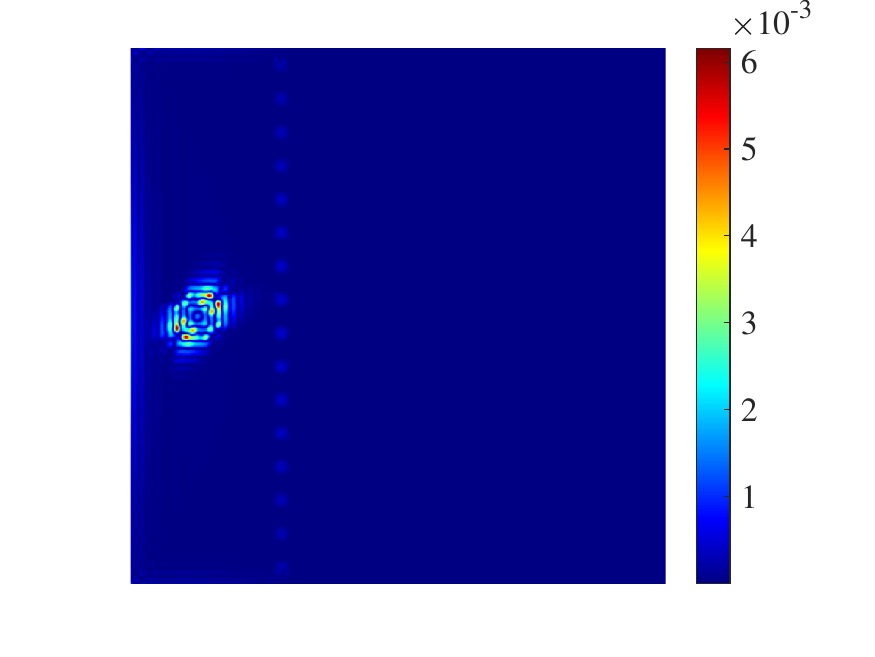}
    \caption{The comparison of the computations without and with the boundary correction. Here the distribution of the absolute error $|u - u_{H, \ell}|$ is depicted.}
    \label{fig:comparsion_msfem_boundary_correction}
  \end{figure}

\end{example}

\begin{example}
  In this experiment, the convergence rates of the localized method with boundary corrector are checked for both the 1D and 2D cases. In particular, we choose $n(x) = \exp(-2\cos(\pi x))$ and $f(x) = 1$ over the interval $[0, 1]$. Here we set $\kappa = 2^6$, and compute the reference solution using the standard FEM with $h = 2^{-14}$. Meanwhile, for the 2D problem, we choose $n(x,y) = 1 + \sin(\pi x)\cos(\pi y)$ and $f(x, y) = 1$. We set $\kappa = 2^4$ and $h = 2^{-10}$ to calculate the reference solution. The numerical convergence rates of the localized method with different oversampling sizes $\ell$ are depicted as in \Cref{fig:localized-convergence-rate}.
  \begin{figure}[htbp]
    \centering
    \includegraphics[width=0.45\linewidth]{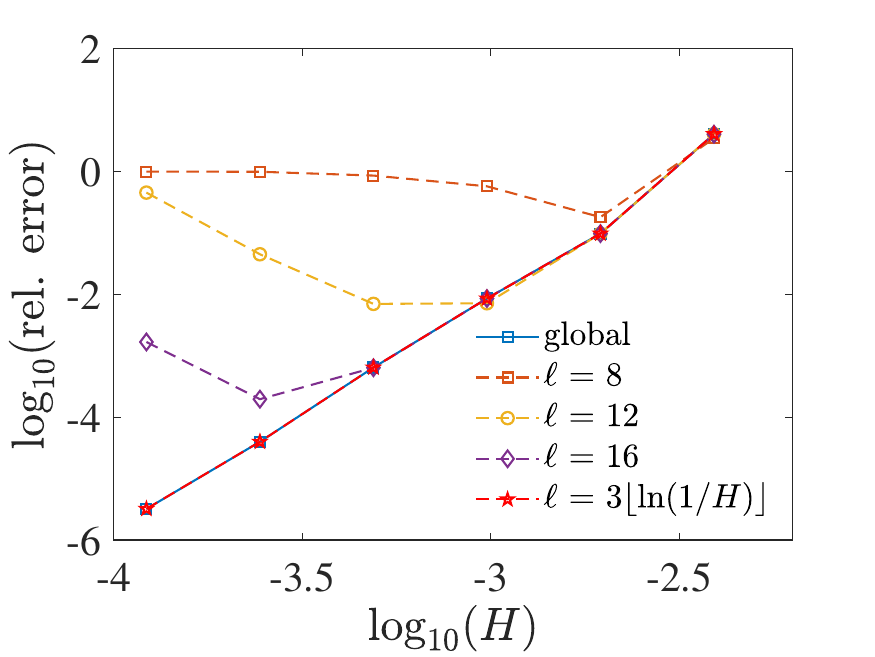}
    \includegraphics[width=0.45\linewidth]{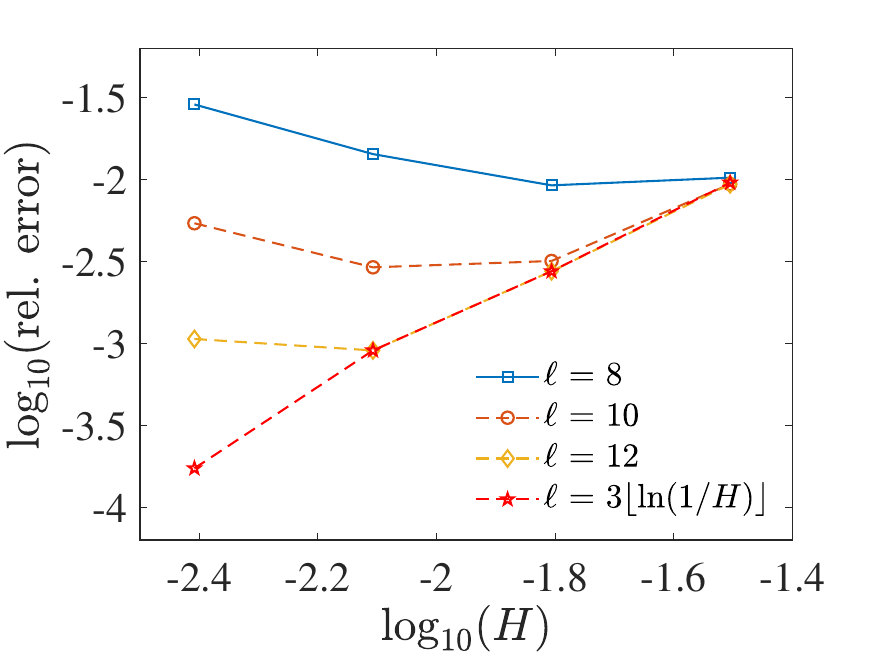}
    \caption{The convergence rates of the relative $V$-error with different oversampling sizes. Here the optimal convergence rates are $\mathcal{O}(H^{4})$ for 1D and $\mathcal{O}(H^{2})$, respectively.}
    \label{fig:localized-convergence-rate}
  \end{figure}

\end{example}

\subsection{Random problems}
In this part, we test the convergence rate of the qMC method and the truncation error. Over the domain $D = [0, 1]$, consider
\begin{equation*}
  n(x, \boldsymbol{\omega}_s) = n_0(x) + \delta\sum_{j = 1}^s \frac{\sin(j\pi x)}{1 + (j\pi)^q}\omega_j,
\end{equation*}
    where $\delta$ controls the strength of the randomness. Here we choose $n_0(x) = 1.0 + 5\sin(\pi x/2)$ and $\delta = 0.5$. And it is clear that for all $j > 0$,
\begin{equation*}
    \|\psi_j(x)\|_{\infty} = \frac{1}{1 + (j\pi)^q} < \frac{1}{(j\pi)^q}.
\end{equation*}
Hence we have $\sum_{j = 1}^{\infty} \|\psi_j\|_{\infty} < \zeta(q) / \pi^q$. In turn, we can take $p \in (1/q, 1)$ described in \Cref{assump:for-KL-series} and obtain the convergence rate with respect to the truncation dimension $s$ as declared in \Cref{thm:approximation-error-truncation}.  More discussion about this can be found in \cite{gilbert2019analysis}.

We now let $q = 2, 3$ and fix $\kappa = 2^6$, the fine mesh $h = 2^{-14}$, the coarse mesh $H = 2^{-9}$ and the oversampling size $\ell = -3\lfloor \log(H) \rfloor$. Meanwhile, we vary the truncation dimensions $s = 16, 32, \cdots, 256$; and the number of Sobol sequences $N = 500, 800, 1000, 4000, 8000$.

Since the 1D problem with the sufficiently small $h$ is considered here, we choose $s = 64$ and $N = 50000$ to calculate the reference solution to test the convergence rate of the qMC method. The almost first-order convergence rate can be observed as in \Cref{fig:convergence-random-field-1D-qMC}. Meanwhile, to test the convergence rate with respect to $s$, we set $s = 512$ and choose 8000 qMC samples to calculate a reference solution, and the corresponding results are shown in \Cref{fig:convergence-random-field-1D-truncation}. These numerical experiments confirm the error analysis of \Cref{thm:main-result}.
\begin{figure}[htbp]
    \centering
    \subfloat[QMC error w.r.t. $N$.]{\includegraphics[width=0.45\linewidth]{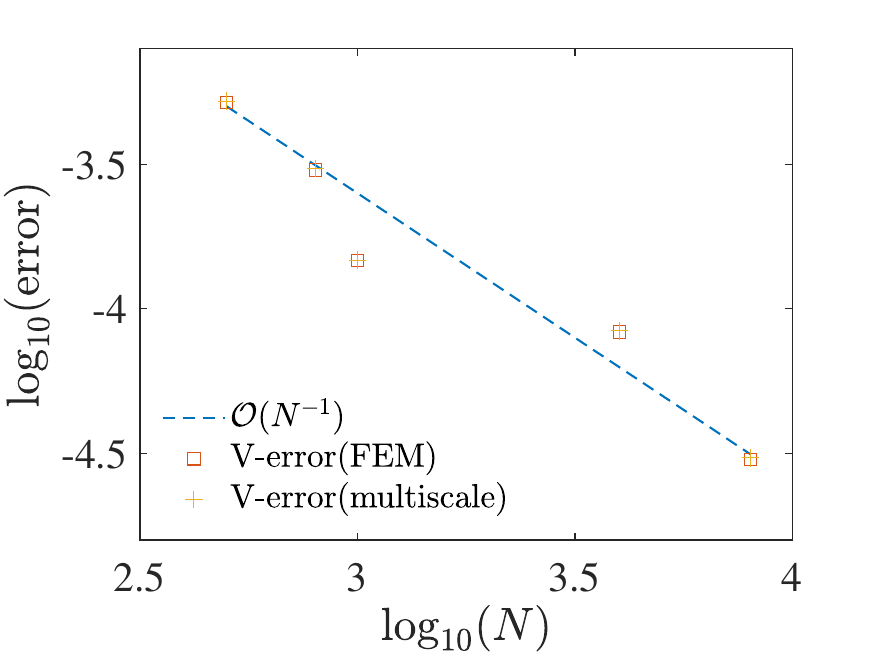}\label{fig:convergence-random-field-1D-qMC}}
    \subfloat[Truncation error w.r.t. $s$.]{\includegraphics[width=0.45\linewidth]{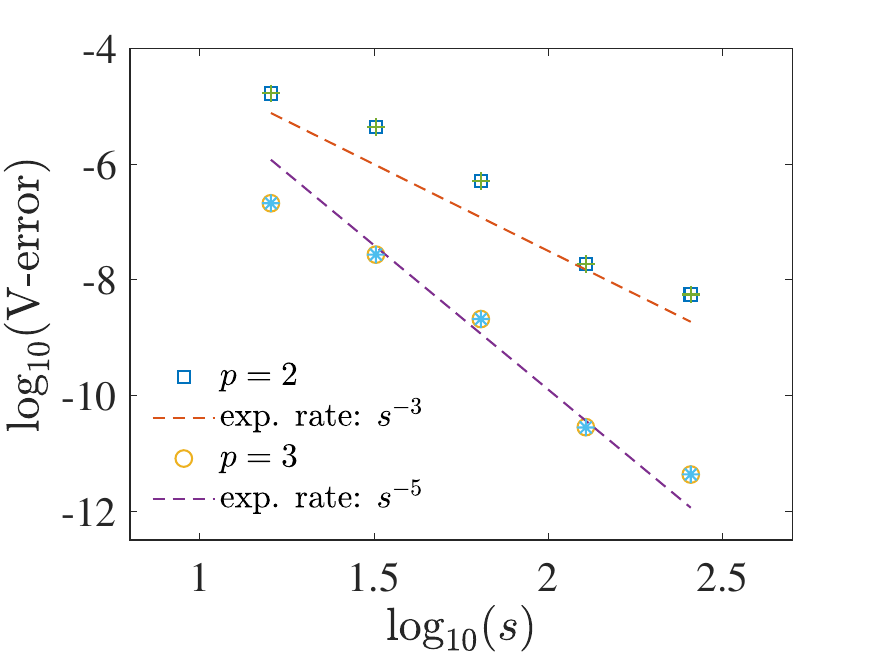}\label{fig:convergence-random-field-1D-truncation}}
    \caption{The numerical convergence rates with respect to the qMC sample size $N$ and the truncation dimension $s$. Note that the FEM solution are also depicted for comparison, and in \Cref{fig:convergence-random-field-1D-truncation}, the corresponding results are labeled squares and circles.}
    \label{fig:convergence-random-field-1D}
\end{figure}

\section{Conclusions}
\label{sec:conclusions}
In this paper, we develop an accurate numerical approach for simulating the Helmholtz problem in random media, where the randomness is represented by a truncated series with components parameterized by i.i.d. stochastic variables. As the truncation dimension increases, such problems typically suffer from mesh-dependent conditions that can significantly limit the computable problems. To alleviate the computational burden and eliminate the pollution effect inherent to the Helmholtz problem, we introduce a multiscale method incorporating a boundary corrector for cases involving Robin boundary conditions. By further employing the quasi-Monte Carlo sampling method, our approach demonstrates weak dependence on the wave number and achieves high convergence rates in both physical and stochastic spaces. Rigorous theoretical analysis and comprehensive numerical experiments are conducted to validate and illustrate the effectiveness and main features of the proposed method.

\section*{Acknowledgments} The research of Z. Zhang is supported by the National Natural Science Foundation of China (projects 12171406 and 92470103), the Hong Kong RGC grant (projects 17307921 and 17304324), the Outstanding Young Researcher Award of HKU (2020-21), and seed funding from the HKU-TCL Joint Research Center for Artificial Intelligence. The simulations are performed using research computing facilities offered by Information Technology Services, University of Hong Kong.


\bibliographystyle{siamplain}
\bibliography{references}
\end{document}